\documentclass[11pt,reqno,a4paper]{amsart}
\usepackage[british]{babel}
\usepackage{amssymb,amstext,amsthm,eucal, float, upgreek, stmaryrd}
\usepackage{latexsym,amsmath,amsfonts,amscd,mathrsfs, booktabs}
\usepackage{enumerate}
\usepackage{graphicx}
\usepackage{caption}
\usepackage{subcaption}
\usepackage{array,color}
\usepackage{enumitem}
\usepackage{tgpagella}
\usepackage{tikz, bbm}
\usepackage[utf8]{inputenc}
\usepackage[margin=2.8cm]{geometry}
\usepackage[small]{eulervm}
\usepackage[unicode]{hyperref}
\usepackage{longtable}
\usetikzlibrary{arrows ,decorations.markings,positioning}
\tikzset{inner sep=0pt, node distance=5mm,
  root/.style={circle,draw,minimum size=5pt,thick},
  broot/.style={circle,draw,minimum size=5pt,thick,fill},
  xroot/.style={circle,draw,minimum size=5pt,thick,label=below:$\times$},
  doublearrow/.style={postaction={decorate},   decoration={markings,mark=at position .6 with {\arrow[line width=1.2pt]{>}}},double distance=1.6pt,thick},
  rdoublearrow/.style={postaction={decorate},   decoration={markings,mark=at position .4 with {\arrowreversed[line width=1.2pt]{>}}},double distance=1.6pt,thick},
	rtriplearrow/.style={postaction={decorate},   decoration={markings,mark=at position .4 with {\arrowreversed[line width=1.2pt]{>}}},double distance=2.5pt,thick},
	ltriplearrow/.style={postaction={decorate},   decoration={markings,mark=at position .6 with {\arrow[line width=1.2pt]{>}}},double distance=2.5pt,thick},
  curvedline/.style={bend=right}
} 
\hypersetup{
  pdftitle   = {On Jordan classes for Vinberg's \texorpdfstring{$\theta$}{}-groups},
  pdfkeywords = {},
  pdfauthor  = {Giovanna Carnovale, Francesco Esposito and Andrea Santi},
  pdfcreator = {\LaTeX\ with package \flqq hyperref\frqq}
}
\numberwithin{equation}{section}
\theoremstyle{plain}
\newtheorem{theorem}{Theorem}[section]
\newtheorem{proposition}[theorem]{Proposition}
\newtheorem{corollary}[theorem]{Corollary}
\newtheorem{lemma}[theorem]{Lemma}
\theoremstyle{definition}
\newtheorem{rem}[theorem]{Remark}
\newtheorem{example}[theorem]{Example}
\newtheorem{definition}[theorem]{Definition}

\newcommand{\Id}{\operatorname{Id}}

\newcommand{\ad}{\operatorname{ad}}

\newcommand{\fsl}{\mathfrak{sl}}
\newcommand{\fzeta}{\mathfrak{z}}

\newcommand{\fm}{\mathfrak{m}}
\newcommand{\fc}{\mathfrak{c}}
\newcommand{\fp}{\mathfrak{p}}
\newcommand{\fs}{\mathfrak{s}}
\newcommand{\fr}{\mathfrak{r}}

\newcommand{\fg}{\mathfrak{g}}
\newcommand{\fh}{\mathfrak{h}}

\newcommand{\1}{\mathbbm{1}}

\newcommand{\ZZ}{\mathbb{Z}}

\newcommand{\NN}{\mathbb{N}}
\newcommand{\CC}{\mathbb{C}}

\renewcommand{\k}{\kappa}
\newcommand{\rk}{\operatorname{rk}}

\newcommand{\be}{\boldsymbol{e}}

\renewcommand{\O}{\mathcal O}
\newcommand{\Oc}{\mathcal O}

\allowdisplaybreaks
\begin{document}
\title[Jordan classes for $\theta$-groups]{On Jordan classes for Vinberg's $\theta$-groups}
\author{Giovanna Carnovale}
\author{Francesco Esposito}
\author{Andrea Santi}
\address{(G.C. and F. E.) Dipartimento di Matematica ``Tullio Levi-Civita" (DM), Via Trieste, 63 - 35121 Padova, Italy.
(A.S) Department of Mathematics and Statistics, UiT the Arctic University of Norway, 9037
Tromsø Norway.}
\email{carnoval@math.unipd.it, esposito@math.unipd.it, asanti.math$@$gmail.com}
\thanks{}
\begin{abstract}
Popov has recently introduced an analogue of Jordan classes (packets, or decomposition classes) for the action of a $\theta$-group $(G_0,V)$, showing that they are finitely-many, locally-closed, irreducible unions of $G_0$-orbits of constant dimension partitioning $V$.  We carry out a local study of their closures showing that Jordan classes are smooth and that their closure is a union of Jordan classes. We parametrize  Jordan classes  and $G_0$-orbits in a given class in terms of the action of subgroups of Vinberg's little Weyl group, and include several examples and counterexamples underlying the differences with the symmetric case and the critical issues arising in the $\theta$-situation.
\end{abstract}
\maketitle
\vskip-0.5cm\par\noindent
\tableofcontents
\vskip-0.5cm\par\noindent
\section{Introduction}
\label{sec:introduction}
\vskip0.3cm\par
Theta groups (or, equivalently,  periodically  graded reductive Lie algebras) were deeply studied in \cite{MR0430168,MR549013} as a natural generalisation of symmetric spaces,  \cite{MR158024,MR0311837}. In all situations $\fg$ is a ${\mathbb Z}_m$-graded complex reductive Lie algebra, its  degree $0$ part $\fg_0$ is again reductive and the focus is on the action of the corresponding connected algebraic group $G_0$ on the other homogeneous components $\fg_i$ of $\fg$. As observed by Vinberg, there is no loss of generality in studying the action on the degree $1$ component $V=\fg_1$ only. Key  results in \cite{MR0430168} concern n \ theory and include: the introduction of a little Weyl group and  the analogue of the Steinberg map and Chevalley's restriction theorem and the proof that the little Weyl groups are complex reflection groups. These results were confirmed also in positive characteristic, \cite{levy}, where an alternative description of the little Weyl group in terms of the usual Weyl group is proposed.  Many interesting examples in representation theory can be interpreted in terms of graded Lie algebras: for instance, if $\fg$ is the Lie algebra of a classical group $G$, a grading on the defining representation of $G$ induces a grading on $\fg$ and the $G_0$-action on $V$ can be seen as a representation of a cyclic quiver with additional structure, \cite[\S 0.3, \S9.5]{LY1}.

A structural feature of theta groups is that they are visible groups, that is, connected reductive algebraic groups with a representation for which  each fiber of the Steinberg map consists of finitely many orbits. This property almost characterizes the theta groups: more precisely, a connected simple irreducible visible linear group is either a (commutant of a) theta group or it is isomorphic to $\mathrm{Spin}(11)$ or $\mathrm{Spin}(13)$ \cite{MR575790}. Various explicit descriptions of the orbits and invariants for theta groups of order  $m=2$ are known (see \cite[Summary Table]{MR1100485}) but a number of cases with $m\geq 3$ have also been considered in the literature \cite{MR2765379,MR1680007,  MR1773770, MR1789188, MR504529}. 

An important application of theta group theory is in the representation theory of reductive groups over a $p$-adic field $F$. Indeed, the classification of positive rank gradings  \cite{levy,levy2,reeder} over the residue field $k$ of a maximal unramified extension $L$ of $F$ leads to the classification of non-degenerate $K$-types, and stable $G_0$-orbits in $V^*$ are stricly related to supercuspidal representations of the rational points of $G$ over $F$ attached to elliptic Z-regular elements of the Weyl group, \cite{reeder2}.   Also, in the context of a graded version of the generalized Springer correspondence, the block decomposition of the $G_0$-equivariant derived category supported on the nilpotent part of each $\fg_i$  leads to the construction of representations of various graded double affine Hecke algebras with possibly unequal parameters, one for each block, \cite{LY1, LY3}. It emerges from these constructions that parabolic induction is no longer the right instrument in the graded setting, leading to the introduction of spirals.  This shows that even though many results in the classical symmetric case have an analogue in the graded setting, generalisations to the case of $m>2$ are not always straightforward. 

This phenomenon is also visible in the study of related $G_0$-stable stratifications in $V$.  In the setting of the ungraded generalized Springer correspondence, one of the relevant stratifications is given by the decomposition into Jordan classes  (packets, or decomposition classes) in a reductive group $G$, or Lie algebra $\fg$. In the Lie algebra setting Jordan classes were introduced in \cite{BK} and were crucial in the construction of sheets for the adjoint action of a semisimple group $G$ on its Lie algebra. These classes are $G$-stable, disjoint, finitely-many, locally-closed, smooth and irreducible. The decomposition into Jordan classes in a Lie algebra turns out to coincide with the decomposition into orbit-types, i.e.,  into the subsets of elements with same stabilizer up to conjugation. Borho and Kraft proved that sheets are easily described as regular closures of those Jordan classes which satisfy some maximality property with respect to closure inclusion, and it was shown in \cite{Bo} that the closure and regular closure of Jordan classes can be described in terms of Lusztig-Spaltenstein's parabolic induction of adjoint orbits. The symmetric analogue of Jordan classes and sheets has been studied by Tauvel and Yu, (see \cite{MR2146652} and references in there) and their closures were studied in \cite{Bu,BH}. In the latter it is again observed that parabolic induction is no longer efficient, and slice induction is proposed: one of the difficulties in working with parabolic induction is the fact that many homogenous Levi subalgebras do not necessarily lie in a homogeneous parabolic subalgebra, see the Appendix \ref{appendix:homogeneous-Levi} for an example of this phenomenon.  
An analogue of Jordan classes for theta groups when $\fg$ is semisimple has been recently introduced by Popov in \cite{MR3890218}, generalizing the classical and symmetric ones. As in these cases, Jordan classes form a partition of $V$ into finitely-many, locally-closed, irreducible unions of $G_0$-orbits of constant dimension, and so sheets for the $G_0$-action on $V$ are regular closures of some Jordan class. 
In this paper we introduce a local study of such Jordan classes and their closures leading us to prove that any Jordan class is smooth and that its closure is a union of Jordan classes. In order to characterize the closure relation,  we provide an analogue of the results in \cite{BH} on slice induction. For our inductive arguments, we needed to extend slightly the notion of Jordan classes to the case of reductive Lie algebras. Our local approach differs from \cite{BH} because we rely on Luna's fundamental Lemma and use the Slodowy slice only after reduction to neighbourhoods of nilpotent points; Luna's slice theorem is also used for the proof of smoothness.

It is also worthwhile to notice that a different, coarser, notion of Jordan equivalence relation could have been introduced, by using regularity for the $G_0$-action rather than for the action of the full group $G$. In the symmetric setting these two notions coincide by virtue of \cite[Proposition 5]{MR0311837}, but  they might differ for $m>2$. Popov's choice of Jordan classes in $V$ ensures that each of them is contained in a usual Jordan class in $\fg$. We devote \S\ref{subsec:reg-conditions} and \S\ref{subsec:regularity questions}
to comparisons of different  notions of regularity and refer to \cite{MR2111215,MR2765379} for an analysis of various results on regularity in theta groups.

The paper is structured as follows. In \S\ref{sec:preliminariestheta}, we recall the basics on periodically graded complex reductive Lie algebras, introduce the relevant notions of regularity and extend to the reductive case the general treatment in \cite{MR3890218} of Jordan classes and sheets in $V$.  We then focus in \S \ref{sec:closure-Jordan-classes} on the local study of the closures of Jordan classes in $V$, the main results here are Theorem \ref{prop:local_holds}, Theorem \ref{thm:union-Jordan-classes}  
and Proposition \ref{prop:Jordan-classes-are-smooth}. We conclude \S\ref{sec:closure-Jordan-classes} with some regularity questions, including Proposition \ref{prop:chiusura-reg-bullet}.
The last section is devoted to slice induction, leading to Theorem \ref{thm:BH}, and to the parametrization of  the Jordan classes in $V$ and the $G_0$-orbits contained in a class. The paper finishes with Example \ref{exampleE8-Jordan-classes} on trivectors in $9$-dimensional space and with Appendix \ref{appendix:homogeneous-Levi}, dealing with obstructions to the existence of homogeneous parabolic subalgebras in $\fg$.

The results of this paper are valid over any algebraically closed field of characteristic zero, but we expect that they extend to suitable positive characteristics. Indeed the results from \cite{luna} needed in \S\ref{sec:closure-Jordan-classes}-\S\ref{sec:slice-parametrization} hold also in positive characteristic under mild restrictions,  cf. \cite{BR}.

During completion of this paper we were informed that Professor \`E. B. Vinberg had passed away. Without his work in \cite{MR0430168} this manuscript would never have been written, so we would like to dedicate it to his memory.

\section{Preliminaries on Vinberg's \texorpdfstring{$\theta$}{}-groups and Jordan classes}
\label{sec:preliminariestheta}
\subsection{Graded Lie algebras}
\label{subsec:gla}
Let $\fg$ be a complex reductive Lie algebra which is $\ZZ_m$-graded, that is, it admits a direct sum decomposition of vector spaces
\begin{equation}
\label{eq:gla}
\fg=\bigoplus_{l\in\ZZ_m}\fg_{l} 
\end{equation}
with $[\fg_{i},\fg_{l}]\subset \fg_{i+l}$ for all $i,l\in\ZZ_m$. We note that the subspaces of 
\eqref{eq:gla} can be recovered as the eigenspaces 
of the automorphism
$\theta$ of $\fg$ 
defined by
$\theta(x)=\omega^l x$ for $x\in\fg_{l}$, where 
$\omega=e^{\frac{2\pi i}{m}}$. 
Conversely, any automorphism $\theta$ of $\fg$ of period $m$ defines a $\ZZ_m$-grading. 
Due to this, we will denote a Lie algebra $\fg$ with a $\ZZ_m$-grading by the triple $\{\fg,\theta,m\}$, or often simply by $\{\fg,\theta\}$. Whenever a subspace $A\subset \fg$ is homogeneous, i.e., it satisfies $A=\oplus_l (A\cap \fg_l)$, we will write
$A_l=A\cap \fg_l$ and 
$A=\oplus_l A_l$. 

The Lie algebra $\fg$ has a decomposition into homogeneous ideals
\begin{align}\label{eq:reductive}
\fg=\fzeta(\fg)\oplus\fs,\; \textrm{ where }\fs:=[\fg,\fg].\end{align} We denote by $\k$ a bilinear form  on $\fg$ that is non-degenerate, $\fg$-invariant, $\theta$-invariant
and such that $\fzeta(\fg)$ and $\fs$ are orthogonal. We call any such bilinear form {\it adapted}.
\begin{lemma}
\label{lem:adapted-bilinear-form}
There exists an adapted bilinear form $\k$ on $\fg$ if and only if $\fzeta(\fg)$ is symmetrically graded, i.e.,  $\dim\fzeta(\fg)_{l}=\dim\fzeta(\fg)_{-l}$ for all $l\in\mathbb Z_m$. 
In this case $\dim\fg_{l}=\dim\fg_{-l}$ for all $l\in\mathbb Z_m$.
\end{lemma}
\begin{proof}
If $\k$ is adapted, then 
$\k(\fg_{l},\fg_{ i})=0$
whenever $i+ l\neq  0$,
hence $\fg_{- l}$ and $\fg_{ l}$ are dual spaces, and so are
$\fs_{- l}$ and $\fs_{ l}$ and also $\fzeta(\fg)_{- l}$ and $\fzeta(\fg)_{l}$.
In particular $\fzeta(\fg)$ is symmetrically graded. 

Conversely, it is enough to
consider an appropriate extension of the Killing form of $\fs$.
\end{proof}
 
{\it With the term reductive $\ZZ_m$-graded Lie algebra $\{\fg,\theta\}$, we will always mean a complex reductive Lie algebra $\fg=\fzeta(\fg)\oplus\fs$ together with a
$\ZZ_m$-grading such that the center $\fzeta(\fg)$ is symmetrically graded.} This is also the class of graded Lie algebras considered in  \cite{MR0430168}, since they allow for adapted bilinear forms. By Lemma \ref{lem:adapted-bilinear-form} we may assume $\k$ to be an extension of the Killing form of $\fs$.
\vskip0.2cm\par

Let $G$ be any connected algebraic group with Lie algebra $Lie(G)=\fg$, let $S$ be the connected subgroup of $G$ with $Lie(S)=\fs$, and let $^\circ$ denote the identity component of a closed subgroup, so $G=Z(G)^\circ S$. Let $G_0$  be the connected subgroup of $G$ with  $Lie(G_0)=\fg_0$. 
Unless otherwise stated, for Lie subalgebras of $\fg$  we will use a gothic letter, the corresponding Roman capital letter will indicate the  connected subgroup of $G$ with that Lie algebra, a lower index $0$ its intersection with $G_0$. So, the decomposition $\fg_0=\fzeta(\fg)_0\oplus\fs_0$, gives an almost direct product $G_0=Z(G)_0^\circ S_0^\circ$, where $S_0^\circ$ is the reductive, connected subgroup of $S$ with $Lie(S_0^\circ)=\fs_0$.  By restricting the adjoint representation, $G_0$ and $S_0^\circ$ act on $\fg_{l}$, for any $l\in\ZZ_m$, with trivial action of $Z(G)_0^\circ$.  The reduction process in \cite[pag. 467]{MR0430168} shows that it is enough to focus on the case of $l=1$; we set $V:=\fg_1$.
The linear group of transformations of $V$ associated to $G_0$ is called the {\it $\theta$-group} of the graded Lie algebra $\left\{\fg,\theta\right\}$ and  it does not depend on the choice of $G$ in the class of locally isomorphic groups. However, by abuse of notation, we will directly refer to $G_0$ as the $\theta$-group of $\left\{\fg,\theta\right\}$.  
The decomposition \eqref{eq:reductive} in degree $1$ gives a decomposition of $V$ into $G_0$-stable subspaces $V=\fzeta(\fg)_1\oplus \fs_1$ with trivial $G_0$-action on $\fzeta(\fg)_1$. Observe that $\fzeta(\fg)_1\not=0$ may occur only if $\theta$ is not inner.

Let $x\in \fg$ and $\fm$ be a Lie subalgebra of $\fg$ with associated subgroup $M\subset G$. The orbit of $x$ for the action of $M$ is denoted by $\O_x^{M}$, and the stabilizer of $x$ in $M$ by $M^x$. 
The centralizer of $x$ in $\fm$ is denoted by $\fm^x$, with center $\fzeta(\fm^x)$. If $x\in V$, then  $\fg^x$, $\fzeta(\fg^x)$ and $[\fg^x,\,\fg^x]$ are $\theta$-stable, in other words homogeneous.  We recall that if $x\in\fg$ is semisimple, then $G^{x}$ is a connected subgroup of $G$, the Levi subgroup of a parabolic subgroup of $G$ \cite[7.3.5]{MR632835}. In this case, the restriction of $\k$ to $\fg^x=\fzeta(\fg^x)\oplus [\fg^x,\fg^x]$
is an adapted bilinear form, so $\fzeta(\fg^x)=\fzeta(\fg)\oplus \fzeta(\fs^x)$ is symmetrically graded.
We stress that $G_0^{x}=G^{x}\cap G_0$ is not connected in general.

We recall the following general results on centralizers, that we will later apply when $x\in V$.
\begin{lemma}\cite[Proposition 35.3.1, Corollary 35.3.2]{MR2146652}
\label{lem:doublecentr}
Let $x\in \fg$. Then
\begin{align}
\label{eq:orthogonal}
[\fg,x]^\perp=[\fg,\fzeta(\fg^x)]^\perp=\fg^x,\; \textrm{ and }
[\fg,\fg^x]^\perp=\fzeta(\fg^x).
\end{align}
The following conditions are equivalent for any $x,y\in \fg$:
\begin{itemize}
	\item[$(i)$] $y\in\fzeta(\fg^x)$;
	\item[$(ii)$] $\fg^x\subset \fg^y$;
	\item[$(iii)$] $[\fg,y]\subset [\fg,x]$;
	\item[$(iv)$] $\fzeta(\fg^y)\subset \fzeta(\fg^x)$;
\end{itemize}
\end{lemma}
\begin{corollary}
\label{cor:preliminaryfacts}
Let $g\in G_0$ and $x,y\in V$. Then the following conditions are equivalent:
\begin{itemize}
	\item[$(i)$] $g\cdot\fg^{x}=\fg^{y}$;
	\item[$(ii)$] $g\cdot \fzeta(\fg^{x})=\fzeta(\fg^{y})$;
	\item[$(iii)$] $g\cdot \fzeta(\fg^{x})_1=\fzeta(\fg^{y})_1$.
\end{itemize}
\end{corollary}
\begin{proof}
Clearly $(i)\Leftrightarrow (ii)$ by Lemma \ref{lem:doublecentr} since $g\cdot\fg^{x}=\fg^{g\cdot x}$ and $g\cdot \fzeta(\fg^{x})=\fzeta(\fg^{g\cdot x})$, and $(ii)\Rightarrow (iii)$. If $(iii)$ holds, then $y\in \fzeta(\fg^{y})_1=\fzeta(\fg^{g\cdot x})_1$ and
$x\in\fzeta(\fg^{g^{-1}\cdot y})_1$, hence $g\cdot\fg^{x}=\fg^y$ by Lemma \ref{lem:doublecentr}.
\end{proof}
\subsection{The Jordan decomposition}
\label{subsec:jd}
Let $\{\fg,\theta\}$ be a reductive $\ZZ_m$-graded Lie algebra.
For elements $x, y, z$ in $\fg$,  lower indices $s$ and $n$ will always indicate semisimple and nilpotent parts in the Jordan decomposition, i.e., they stand for $x=x_s+x_n$ with  $x_s\in\fg$ semisimple, $x_n\in\fg$ nilpotent, and $[x_s,x_n]=0$. 
Elements of $\fzeta(\fg)$ are always intended to be semisimple. 

Let $\mathcal S$ (resp. $\mathcal N$) be the set of semisimple (resp. nilpotent) elements of $\fg$.
We note that $\theta$ preserves both $\mathcal S$ and $\mathcal N$,
so semisimple and nilpotent parts of any $x\in\fg_{ l}$ also belong to $\fg_{ l}$.
We set $\mathcal S_V=\mathcal S\cap V$, $\mathcal N_V=\mathcal N\cap V$, and
stress that the number of $G_0$-orbits in $\mathcal{N}_V$ is finite \cite{MR0430168}. 

\begin{lemma}
The  action of the reductive  Lie algebra $\fg_{0}$  on  $\fg$ is completely reducible.
\end{lemma}
\begin{proof}
Since $\fg_0=\fzeta(\fg)_0\oplus\fs_0$, it is sufficient to prove the claim for $\fs_0$. 
Now $\k$ 
restricted to $\fs_{0}$ is non-degenerate and $\fs_{0}$ contains the semisimple and nilpotent parts of any of its elements. The claim then follows from, e.g., \cite[Proposition 20.5.12]{MR2146652}.
\end{proof}
We emphasize that $\fg_0$ is not a subalgebra of maximal rank of $\fg$ in general, that is, it might not contain any Cartan subalgebra of $\fg$.
Let $x\in V$. A direct consequence of Lemma \ref{lem:doublecentr} is:
\begin{lemma}
\label{lem:tanort} The tangent space $T_x\O^{G_0}_x$ to $\O^{G_0}_x$ at $x$ is given by the subspace $[\fg_{0},x]$ of $V$. Its orthogonal complement in $\fg_{-1}$ coincides with $\fg^x_{-1}$.
\end{lemma}
\subsection{The Cartan subspace}
\label{subsec:cartansubspace}
A {\it Cartan subspace} of $\{\fg,\theta\}$ is an abelian subspace $\fc$ of $V$ which consists of semisimple elements and it is maximal in the class of such subspaces. 

\begin{theorem}\cite[pag. 472]{MR0430168}
\label{thm:472}
Any two Cartan subspaces of $\left\{\fg,\theta\right\}$ are conjugate by the action of an element in $G_0$. As a consequence, if $x\in {\mathcal S}_V$, then 
$\Oc_x^{G_0}$ meets any Cartan subspace of $\{\fg,\theta\}$.
\end{theorem}
The dimension of a Cartan subspace of a graded Lie algebra $\{\fg,\theta\}$ is called the {\it rank} of $\{\fg,\theta\}$. 
It is clear that 
$\left\{\fg,\theta\right\}$ has zero rank if and only if $V\subset {\mathcal N}_V$.
For any set $R$ of commuting elements of $\mathcal{S}_V$, the centralizer $\fc_\fg(R)=\cap_{x\in R}\fg^x$ of $R$ in $\fg$
is a homogeneous Levi subalgebra of $\fg$, so
\begin{equation}
\label{eq:centralizerdecomp}
\fc_\fg(R)=\fzeta(\fc_\fg(R))\oplus [\fc_\fg(R),\fc_\fg(R)]
\end{equation} and these summands are also homogeneous. We recall a useful characterization of a Cartan subspace in terms of its centralizer \cite[pag. 471]{MR0430168}.
\begin{proposition}
\label{prop:Cartansubspaceandcentralizerofc}
A subspace $\fc\subset V$ consisting of commuting semisimple elements is a Cartan subspace if and only if
$\fzeta(\fc_\fg(\fc))_1=\fc$ and the graded Lie algebra $\left\{ [\fc_\fg(\fc),\fc_\fg(\fc)],\theta\right\}$ has zero rank.
\end{proposition}
Let $\fc$ be a Cartan subspace. By the previous result and equation \eqref{eq:centralizerdecomp} for $R=\fc$, we have a decomposition
$\fc_{\fg}(\fc)_1=\fc\bigoplus [\fc_\fg(\fc),\fc_\fg(\fc)]_1$,
with $\fc\subset \mathcal{S}_V$  and  $ [\fc_\fg(\fc),\fc_\fg(\fc)]_1\subset{\mathcal N}_V$. 
In other words, this decomposition gives the Jordan components  of any
element of 
$\fc_{\fg}(\fc)_1$. 
\begin{corollary}
\label{cor:serveanchedopo}
For any $x\in\fc$, we have $\fzeta(\fg^{x})_1\subset\fc$.
\end{corollary}
\begin{proof}
Since
$\fzeta(\fg^{x})$
consists of semisimple elements, it follows that
$\fzeta(\fg^{x})_1\subset \fc_{\fg}(\fc)_1\cap\mathcal{S}_V=\fc$.
\end{proof}
Before turning to the next subsection, we recall that the Weyl group in the sense of Vinberg
is the group $W_{Vin}=W_{Vin}(\fg,\theta)$ of linear transformations of $\fc$ given by
$W_{Vin}\cong N_{G_0}(\fc)/Z_{G_0}(\fc)$,
where $N_{G_0}(\fc)$ (resp. $Z_{G_0}(\fc)$) is the normalizer (resp. centralizer) of $\fc$ in $G_0$.
\begin{theorem}\cite[pag. 473]{MR0430168}
\label{theorem:weylvinberg}
The group $W_{Vin}$ is finite and  for $x,\,y\in \fc$ we have $y\in\Oc_x^{G_0}$ if and only if $y\in W_{Vin}\cdot x$.
\end{theorem}
There is a geometric counterpart to this result \cite[\S 4]{MR0430168}. 
The restriction $\CC[V]\to \CC[\fc]$  of polynomial functions  from $V$ to $\fc$
induces a ``Chevalley-type'' isomorphism $\CC[V]^{G_0}\cong \CC[\fc]^{W_{Vin}}$
and each fiber of the ``Steinberg quotient map'' 
$\varphi\colon V\to V/\!/ G_0\cong \fc/W_{Vin}$
consists of finitely many $G_0$-orbits. Here  $V/\!/ G_0$
is the GIT quotient of $V$, and two elements of $V$ fail to be separated by the invariants if and only if their semisimple parts lie in the same $G_0$-orbit. Recall that semisimple (resp. nilpotent) orbits can also be characterized as the closed orbits (resp. orbits whose closure contains $0$). Hence, each fiber of $\varphi$ contains exactly one closed orbit.

\subsection{Dimensions of centralizers and regularity conditions}
\label{subsec:counterexample}
\label{subsec:reg-conditions}
This subsection deals with some 
general observations, which encompas a classical
result of Kostant and Rallis (see \cite{MR0311837} and also \cite{MR2111215}),
and motivates the introduction of two distinct notions of regularity.
\begin{proposition}
\label{prop:dimensions}
Let $\{\fg,\theta\}$ be a  reductive $\ZZ_m$-graded Lie algebra (with symmetrically graded center, as usual). Then
$\dim\fg_l-\dim\fg_l^x=\dim\fg_{-l-1}-\dim\fg_{-l-1}^x$
for all $x\in V$ and $l\in\ZZ_m$. 
\begin{proof}
Let $\k$ be an adapted bilinear form on $\fg$.
The bilinear form  given by
$\k_x(y,z):=\k(x,[y,z])$
 is skew-symmetric for all $y,z\in\fg$ and its radical is the centralizer 
$
\fg^x
$, 
which is homogeneous. It induces a non-degenerate bilinear form on the quotient
$\fg/\fg^x=\bigoplus_{l\in\ZZ_m}\fg_l/\fg^x_l$
with the property that 
$\fg_i/\fg^x_i\perp\fg_l/\fg^x_l$ if $i+l+1\neq 0$,
in particular $\fg_l/\fg^x_l\cong (\fg_{-l-1}/\fg^x_{-l-1})^*$.
\end{proof}
\end{proposition}
\begin{corollary}
\label{corollaryK-R}
\vskip0.1cm\par\noindent
\begin{itemize}
\item[$(i)$]
For all $x\in V$ we have
$\dim\O^G_x=2\dim\O_x^{G_0}+\sum_{l\neq -1, 0}\big(\!\dim\fg_l-\dim\fg^x_l\big)$;
\item[$(ii)$] 
	If $x\in\mathcal S_V$, then 
$
\dim\fg_l-\dim\fg_l^x=\dim\fg_{l+1}-\dim\fg_{l+1}^x
$
is independent of $l\in\mathbb Z_m$ and we have $\dim\O_x^G=m\dim\O_x^{G_0}$;
\item[$(iii)$] Let $x\in V$, then $\fg^x_0=\fg_0$ if and only if $x\in\fzeta(\fg)_1$.
\end{itemize}
\end{corollary}
\begin{proof}
Claim (i) is immediate from Proposition \ref{prop:dimensions}. 
If $x\in\mathcal S_V$, then the restriction of $\k$ to $\fg^x$ is non-degenerate and $\dim\fg^x_l=\dim\fg^x_{-l}$ for all $l\in\mathbb Z_m$, so (ii) follows from Proposition \ref{prop:dimensions} and (i). If $x\in\fzeta(\fg)_1$, then clearly $\fg^x_0=\fg_0$.
Conversely, if $\fg^x_0=\fg_0$ then $x\in\fc_{\fg}(\fh_0)$, where $\fh_0$ is a Cartan subalgebra of $\fg_0$, and $x$ is semisimple by
a classical result, see e.g. \cite[pag. 116]{MR1349140}. Then $\fg^x=\fg$ by (ii) and $x\in\fzeta(\fg)_1$.
\end{proof}
If $x,y\in V$ are two elements with $\dim\O^{G_0}_x=\dim\O^{G_0}_y$, then $\dim\fg_l^x=\dim\fg_l^y$ for $l=0,-1$. 
The following simple example shows that the hypothesis $x,y\in\mathcal S_V$ is indeed necessary for
$\dim\fg_{l}^x=\dim\fg_{l}^{y}$ to hold also for $l\neq 0,-1$.

\begin{example}
\label{ex:tworegularityconditions}
Let $\fg$ be of type $E_8$ and $\theta$ be the automorphism of $\fg$ of order $3$ extensively studied in \cite{MR504529}.
Here $\fg_{1}\cong\Lambda^3\CC^9$, $\fg_0\cong\fsl(9)$ and 
$\fg_{-1}\cong\Lambda^3(\CC^9)^*$.
The orbits of $\mathrm{SL}(9)$ on $V=\Lambda^3\CC^9$ have been classified in \emph{loc. cit}. Let $\be_i$, for $1\leq i\leq 9$, be the canonical basis vectors of  $\CC^9$ and let $\be_{ijl}:=\be_i\wedge \be_j\wedge \be_l$. The trivector 
$
x_s=\be_{123}+\be_{456}+\be_{789}
$
is semisimple, with centralizer $\fg^{x_s}$ a reductive Lie algebra with semisimple part $\fr$ of type $E_6$. More precisely  $\fr=\fr_{-1}\oplus\fr_0\oplus\fr_1$ with 
$$
\fr_1=X\otimes Y\otimes Z\;,\quad\fr_0=\fsl(X)\oplus\fsl(Y)\oplus\fsl(Z)\;,\quad \fr_{-1}=X^*\otimes Y^*\otimes Z^*\;,
$$
where $X=\mathrm{span}\{\be_1,\be_2,\be_3\}$, $Y=\mathrm{span}\{\be_4,\be_5,\be_6\}$, $Z=\mathrm{span}\{\be_7,\be_8,\be_9\}$ and where we identified tensor products with subspaces of $\fg_{\pm 1}$ by mapping pure tensors to the corresponding antisymmetrizations. Since $\fg^{x_s}$ has maximal rank, its center is two-dimensional and it is not difficult to see that it consists of $x_s\in \fg_1$ and $x_s^*\in\fg_{-1}$.

Now $x_s$ is the semisimple part of trivectors $x=x_s+x_n$ in the VI family, cf. \cite[Table 5]{MR504529}. We consider those trivectors for which $\dim\O^{G_0}_x=76$, i.e., $x=x_s+x_n$ with nilpotent part:
\begin{itemize}
\item[] {\bf Class} $7$: $x_n=\be_{149}+\be_{158}+\be_{167}+\be_{248}+\be_{357}$;
\item[] {\bf Class} $8$: $x_n=\be_{149}+\be_{167}+\be_{258}+\be_{347}$;
\item[] {\bf Class} $9$: $x_n=\be_{147}+\be_{158}+\be_{258}+\be_{269}$.
\end{itemize}
In all the three cases $\dim\fg^{x}_0=4$ and $\dim\fg^{x}_{-1}=8$ by Proposition \ref{prop:dimensions}.
However a direct computation tells us that
$\fg^x_{1}=\left\{y\in\fg^{x_s}_1\mid y\wedge x_n=0\right\}$
has dimension $6$, $8$ and $10$, respectively.
\end{example}
Corollary \ref{corollaryK-R} and Example \ref{ex:tworegularityconditions} motivate the following. 
\begin{definition}
\label{def:regularity}
For any subset $A\subset V$, we set 
\begin{itemize}
	\item[(i)] $A^{reg}=\Big\{x\in A\mid \dim\fg^{x}\leq \dim\fg^{y}\;\text{for all}\;y\in A\Big\}$;
	\item[(ii)] $A^\bullet=\Big\{x\in A\mid \dim\fg^{x}_0\leq\dim\fg^y_0\;\text{for all}\;y\in A\Big\}$. 
\end{itemize}
The subset $A^{reg}$ (resp., $A^\bullet$) is called the regular part (resp., the $G_0$-regular part) of $A$.
\end{definition}
Note that $A^\bullet=\left\{x\in A\mid \dim\fg^{x}_{-1}\leq\dim\fg^y_{-1}\;\text{for all}\;y\in A\right\}$ due to Lemma \ref{lem:tanort}.
A simple relation between the two notions is given by the following.
\begin{lemma}
\label{lemma:regvsbullet}
Let $A$ be an irreducible subset of $V$ (w.r.t. the induced topology). Then
\begin{equation}
\label{eq:regvsbullet}
A^{reg}=\bigcap_{l\in\mathbb Z_m}\left\{x\in A\mid \dim\fg^x_l\leq\dim\fg^y_l\;\text{for all}\;y\in A\right\}
\end{equation}
and so  $A^{reg}\subset A^{\bullet}$ as a Zariski open subset.
\end{lemma}
\begin{proof}
Clearly each subset on the R.H.S of \eqref{eq:regvsbullet} is non-empty and Zariski open in $A$. Since $A$ is irreducible, the (finite) intersection of all such subsets is non-empty, so equal to $A^{reg}$.
\end{proof}

Let $\Sigma$ be the set of restricted roots, that is, the non-zero linear functions on $\fc$ occurring
in the weight space decomposition of the semisimple action of the abelian subalgebra $\fc$ on $\fg$. We write
$\fg=\fc_{\fg}(\fc)\oplus\bigoplus_{\sigma\in\Sigma}\fg_{\sigma}$.

\begin{example}
\label{rem:G0regular-in-V}
A semisimple element $y\in\mathcal S_V$ belongs to $\mathcal S_V^{reg}$ if and only if it lies in $\mathcal S_V^\bullet$ if and only if $\dim\fg^y=\dim\fc_\fg(\fc)$. The first equivalence follows form   Corollary \ref{corollaryK-R} (ii). For the second one, observe that $y$ is
$G_0$-conjugated to some $x\in\fc$, whose
centralizer is stabilized by $\fc$ and decomposes as
\begin{equation}
\label{eq:centralizerss}
\fg^x=\fc_{\fg}(\fc)\oplus\bigoplus_{\sigma\in\Sigma(x)}\fg_{\sigma}
\end{equation}
with $\Sigma(x)=\left\{\sigma\in\Sigma\mid\sigma(x)=0\right\}$. Hence, 
$x\in\mathcal S_V^{reg}$ if and only if $\sigma(x)\neq 0$ for all $\sigma\in\Sigma$, i.e., $\fg^x=\fc_\fg(\fc)$.
\end{example}
\begin{example}
\label{ex:elements-general-position}
Contrarily to the ungraded case and the $(m=2)$-case, an element $x_s\in\fc\cap\mathcal S_V^{reg}$ is not necessarily in $V^{\bullet}$
(let alone $V^{reg}$ or $\fg^{reg}$, since $\fg^{reg}\cap V\subset V^{reg}\subset V^{\bullet}$).
In general, $x_s$ extends to an element $x=x_s+x_n\in V^{\bullet}$ 
where $x_n$ is  an element in general position in $[\fc_{\fg}(\fc),\fc_{\fg}(\fc)]_1$
(recall that $[\fc_{\fg}(\fc),\fc_{\fg}(\fc)]_1$ consists of nilpotent elements). Then, if $[\fc_{\fg}(\fc),\fc_{\fg}(\fc)]_1\neq0$, also $x_n\neq0$ and so
$\fg_0^x\subsetneq \fg^{x_s}_0$ due to (iii) of Corollary \ref{corollaryK-R} applied to the reductive Lie algebra $\fg^{x_s}$. 
The $G_0$-orbits in $V^\bullet$ have codimension in $V$ equal to the rank of $\{\fg,\theta\}$, hence
$\dim\fg_0^x=\dim\fg_0-\dim V+\dim\fc$, see \cite[Theorem 5]{MR0430168}.
\end{example}
\subsection{Jordan classes and sheets for \texorpdfstring{$\theta$}{}-groups}
\label{subsec:Jordanclasses}
V.\ L.\ Popov has recently generalized the notion of a Jordan class to the case of semisimple $\ZZ_m$-graded Lie algebras $\{\fg,\theta\}$
and studied its main geometric properties in \cite{MR3890218}. For $m=1,2$, the notion coincides with that studied in \cite{BH, MR2146652}.
We here briefly extend his general treatment to the reductive case, which is more suitable for our inductive and local arguments of \S\ref{sec:closure-Jordan-classes}-\S\ref{sec:slice-parametrization}, and directly refer to \cite[\S 3]{MR3890218} for more details.
(We warn the reader that the symbol ``reg'' in \cite{MR3890218} is replaced by ``$\bullet$'' in the present paper.)

Let $\{\fg,\theta\}$ be a  reductive $\ZZ_m$-graded Lie algebra.
Two elements $x=x_s+x_n$ and $y=y_s+y_n$ of $V$
are {\it $G_0$-Jordan equivalent} if there exists $g\in G_0$ such that
\begin{equation}
\label{eq:JCV}
\fg^{y_s}=g\cdot\fg^{x_s}\;, \quad y_n=g\cdot x_n\,.
\end{equation}
This is an equivalence relation $x\overset{G_0}{\sim}y$ on $V$, the equivalence class $J_{G_0}(x)$ of $x\in V$ is called the {\it $G_0$-Jordan class} of $x$ in $V$. Evidently the union of all $G_0$-Jordan classes in $V$ is a partition of $V$.
\begin{rem}
\label{rem:definition-Jordan-classes}
\begin{enumerate}
\item By construction any $G_0$-Jordan class is a $G_0$-stable set consisting of $G_0$-orbits of the same dimension. For example $\mathcal S_V^{reg}$ constitutes a $G_0$-Jordan class, as it can be easily seen from Theorem \ref{thm:472} and Example \ref{rem:G0regular-in-V}. 
\item\label{item:centro} The equality  $\fg^{x_s}=\fg^{z+x_s}$ for any $z\in\fzeta(\fg)_1$ and $x\in\fg_1$ implies that $z+x \overset{G_0}{\sim}x$, so  the additive group underlying $\fzeta(\fg)_1$ acts on each $G_0$-Jordan class $J_{G_0}(x)$ by translations. 
\item Since $G_0=Z(G)_0^\circ S_0^\circ$, the element $g$ from  \eqref{eq:JCV} can always be chosen in $S_0^\circ$. Then, for $x=z+x'\in \fzeta(\fg)_1\oplus\fs_1$ and $y=w+y'\in \fzeta(\fg)_1\oplus\fs_1$, 
the statement $x\overset{G_0}{\sim}y$ holds if and only if $x'\overset{S^\circ_0}{\sim}y'$ holds and the decomposition of $V=\fzeta(\fg)_1\oplus \fs_1$ induces a decomposition 
\begin{equation}\label{eq:reduction}J_{G_0}(x)=J_{G_0}(x')=\fzeta(\fg)_1\times J_{S^\circ_0}(x')\end{equation}
where $J_{S^\circ_0}(x')$ is the $S_0^\circ$-Jordan class of $x'\in\fs_1$ as introduced in \cite{MR3890218}. 
\item Equality \eqref{eq:reduction} applied to $x'\in \mathcal N_V\subset\fs_1$ gives $J_{G_0}(x')=\fzeta(\fg)_1\times\O^{G_0}_{x'}= \fzeta(\fg)_1\times\O^{S^\circ_0}_{x'}$. For $z\in\fzeta(\fg)_1$ we then get $J_{G_0}(z)=J_{G_0}(0)=\fzeta(\fg)_1$. 
\end{enumerate}
\end{rem}

Observe that if $x=x_s+x_n\in V$, then $x_n$ lies in the degree $1$ component of the homogeneous semisimple subalgebra $[\fg^{x_s},\,\fg^{x_s}]$.
\begin{lemma}
\label{lem:doublecentrII}
We have
$\fzeta(\fg^x)=\fzeta(\fg^{x_s})\oplus\fzeta(\fg^{x_n}\cap [\fg^{x_s},\fg^{x_s}])$
and the components of an element in $\fzeta(\fg^x)$ with respect to this decomposition coincide with its semisimple and nilpotent parts, respectively. 
Thus,  $\fzeta(\fg^x)_1=\fzeta(\fg^{x_s})_1\oplus\fzeta(\fg^{x_n}\cap [\fg^{x_s},\fg^{x_s}])_1$.
\end{lemma}
\begin{proof} The first claim is \cite[Proposition 39.1.1]{MR2146652}, the second follows since $\fzeta(\fg^x)$ and its summands are homogeneous. 
\end{proof}
Lemma \ref{lem:doublecentr} tells us that 
\begin{equation}
\label{eq:fzetareg}
\begin{aligned}
(\fzeta(\fg^x)_1)^{reg}
&=\big\{y\in \fzeta(\fg^x)_1\mid \fg^y=\fg^x\big\}\\
&=\big\{y\in \fzeta(\fg^x)_1\mid \fzeta(\fg^y)=\fzeta(\fg^x\big)\}\\
&=\big\{y\in \fzeta(\fg^x)_1\mid \rk(\ad_\fg(y))=\rk(\ad_\fg(x))\big\}\;,
\end{aligned}
\end{equation}
which is a Zariski open subset of $\fzeta(\fg^x)_1$, hence irreducible. We note that this is also the set of
all $y\in V$ such that $\fg^y=\fg^x$ and that $x\in (\fzeta(\fg^x)^{reg})_1$, so $(\fzeta(\fg^x)^{reg})_1=(\fzeta(\fg^x)_1)^{reg}$ and we will omit the parentheses in the sequel. 

The proof of the following result is as in \cite[Lemma 39.1.2 \& Proposition 39.1.5]{MR2146652}, once the last claim of Lemma \ref{lem:doublecentrII} is taken into account. See also \cite[Proposition 3.10]{MR3890218}.
 \begin{proposition}
\label{eq:JordanIrr}
Let $x=x_s+x_n\in V$. Then:
\begin{itemize}
	\item[$(i)$] the decomposition in Lemma \ref{lem:doublecentrII} induces a decomposition
$\fzeta(\fg^x)^{reg}_1=\fzeta(\fg^{x_s})^{reg}_1\times\fzeta(\fg^{x_n}\cap [\fg^{x_s},\fg^{x_s}])^{reg, G^{x_s}}_1
$, where "reg, $G^{x_s}$" is the regular part for the action of $G^{x_s}$;
	\item[$(ii)$] the $G_0$-Jordan class of $x$ is the irreducible subset of $V$ given by
$J_{G_0}(x)=G_0\cdot(\fzeta(\fg^{x_s})^{reg}_1+x_n)$.
\end{itemize}  
\end{proposition}

We will need the following results from \cite{MR3890218} which readily generalize to the reductive case by virtue of \eqref{eq:reduction}.
\begin{proposition}(\cite[Proposition 3.9 and Proposition 3.17]{MR3890218}).
\label{thm:fullcentralizer}
Let $\{\fg,\theta\}$ be a  reductive $\ZZ_m$-graded Lie algebra and $x,y\in V$. Then the following conditions are equivalent:
\begin{itemize}
	\item[$(i)$] $x\overset{G_0}{\sim}y$;
	\item[$(ii)$] there exists $g\in G_0$ such that $\fg^y=g\cdot\fg^x$; 
	\item[$(iii)$] there exists $g\in G_0$ such that $\fzeta(\fg^y)=g\cdot(\fzeta(\fg^x))$.
\end{itemize}
Moreover the number of $G_0$-Jordan classes in $V$ is finite.
\end{proposition}
%
%
\begin{corollary}
\label{cor:JClc}
The $G_0$-Jordan class of $x\in V$ coincides also with
$
J_{G_0}(x)=G_0\cdot\fzeta(\fg^{x})^{reg}_1,
$
it is locally closed in $V$ (hence a subvariety of $V$) and $\dim J_{G_0}(x)=\dim \fg_0-\dim \fg_0^{x}+\dim\fzeta(\fg^{x_s})_1$.
\end{corollary}
\begin{proof}
The first two statements
can be proved as in \cite[Corollary 39.1.7]{MR2146652}, for the last one see \cite[Proposition 3.13]{MR3890218}. 
\end{proof}

It follows from Corollary \ref{cor:JClc} that any $G_0$-Jordan class $J_{G_0}(x)=G_0\cdot\fzeta(\fg^{x})^{reg}_1$ is contained in the $G$-Jordan class $J_{G}(x)=G\cdot \fzeta(\fg^{x})^{reg}$. However it is well-known that two elements $x,y\in V$ in the same $G$-Jordan class
are not $G_0$-Jordan equivalent in general (see \cite[38.7.18]{MR2146652} for an example of nilpotent elements in $V$ that are $G$-conjugate but not $G_0$-conjugate).
We conclude this subsection recalling the relationship between the sheets for the $G_0$-action on $V$ and the $G_0$-Jordan classes. 

Let  $H$ be a connected algebraic group acting on a variety $X$ and let $d\in\NN$. We set $X_{(d)}=\{x\in X~|~ \dim \O_x^H
=d\}$ and for any subset $A\subset X$ we set $A_{(d)}={A}\cap X_{(d)}$. Each $X_{(d)}$ is locally closed and its irreducible components are called {\it sheets} for the $H$-action on $X$.
We observe that $X_{(\leq d)}:=\bigcup_{j\leq d}X_{(j)}$ is closed so $\overline{X_{(d)}}\subset X_{(\leq d)}$ \cite[Proposition 21.4.4]{MR2146652}.

 If $A\subset V$, and $p$ is the largest integer with $A_{(p)}\neq\emptyset$ then, according to Definition \ref{def:regularity}, we have
$A_{(p)}=A^\bullet$,
which is a Zariski open subset of $A$. In particular, the set $V^{\bullet}$ is a Zariski open subset of $V$, hence it is irreducible, and it is called the $G_0$-regular sheet of $V$.
\begin{proposition}\label{prop:dense}(\cite[Proposition 3.19]{MR3890218}) For any sheet $S$ in $V$ there exists a unique $G_0$-Jordan class $J\subset S$ such that $S=\overline{J}^\bullet$. Moreover we have $\overline{S}=\overline{J}$. 
\end{proposition}

\section{Closure of a $G_0$-Jordan class}
\label{sec:closure-Jordan-classes}

\subsection{Closure of $G_0$-Jordan classes: the semisimple parts}
By virtue of Proposition \ref{prop:dense}, it is important to understand the closure and $G_0$-regular closure of a $G_0$-Jordan class and to see which 
classes are dense in a sheet. We start with a preliminary result and then describe which semisimple parts occur in the closure of a $G_0$-Jordan class.

We first describe the image of the closure of a Jordan class through the Steinberg map $\varphi\colon V\to  V/\!/ G_0\cong \fc/W_{Vin}$.  

\begin{lemma}\label{lem:image-restricted-Steinberg}Let $J=J_{G_0}(x)$ be a $G_0$-Jordan class in $V$, with $x_s\in\fc$.  Then,  
$\varphi(\overline{J})=\frac{W_{Vin}\cdot\fzeta(\fg^{x_s})_1}{W_{Vin}}$.
In particular,  it is a closed subset of $\fc/W_{Vin}$, i.e., an affine variety.
\end{lemma}
\begin{proof} Observe that $\fzeta(\fg^{x_s})_1\subset\fc$,  cf. Corollary \ref{cor:serveanchedopo},  so the expression for the image makes sense.  Now,  $\varphi(J)=\varphi(\fzeta(\fg^{x_s})^{reg}_1)$ by Proposition \ref{eq:JordanIrr}, $G_0$-equivariance and \cite[Theorem 3]{MR0430168}, hence
\begin{equation}
\label{eq:usefulI}
\begin{aligned}
\varphi(\overline J)\subset \overline{\varphi(J)}&=
\overline{\varphi(\fzeta(\fg^{x_s})^{reg}_1)}=\frac{\overline{W_{Vin}\cdot\fzeta(\fg^{x_s})^{reg}_1}}{W_{Vin}}
=\frac{W_{Vin}\cdot\fzeta(\fg^{x_s})_1}{W_{Vin}}\;.
\end{aligned}
\end{equation}
On the other hand, if $y_s\in\fzeta(\fg^{x_s})^{reg}_1$, then $y=y_s+x_n\in\Oc_y^{G_0}\subset J$ and so $y_s\in\overline{\Oc_y^{G_0}}\subset\overline{J}$, \cite[Proposition 4]{MR0430168}, giving $\fzeta(\fg^{x_s})^{reg}_1\subset \overline J$. 
It follows that 
\begin{equation}
\label{eq:usefulII}
\fzeta(\fg^{x_s})_1=\overline{\fzeta(\fg^{x_s})^{reg}_1}\subset \overline J\;,
\end{equation}
hence $W_{Vin}\cdot\fzeta(\fg^{x_s})_1\subset \overline J$ and 
$$\frac{W_{Vin}\cdot\fzeta(\fg^{x_s})_1}{W_{Vin}}=\varphi(W_{Vin}\cdot\fzeta(\fg^{x_s})_1)\subset \varphi(\overline J).$$ 
\end{proof}

Let $J=J_{G_0}(x)\subset V_{(d)}$ be a $G_0$-Jordan class in $V$. Then its closure $\overline{J}$ is a union of $G_0$-orbits and if $\Oc_y^{G_0}\subset \overline{J}$, then  $\overline{\Oc_y^{G_0}}\subset \overline{J}$. 
Let $\mathcal{M}_{\overline J}$ be the set of  $G_0$-orbits contained in $\overline{J}$ which are maximal with respect to the partial order given by inclusion of orbit closures.  By construction $\overline{J}=\bigcup_{\Oc\in{\mathcal M}_{\overline J}}\overline{\Oc}$.

The following proposition can be seen as a partial analogue of \cite[\S3.1]{Bo}.
\begin{proposition}
\label{prop:maxorbI}
Let $J=J_{G_0}(x)$ be a $G_0$-Jordan class in $V$. Then 
$\overline{J}^\bullet=\bigcup_{\Oc\in{\mathcal M}_{\overline J}}\Oc$.  
\end{proposition}
\begin{proof} 
We may assume without loss of generality that $x=x_s+x_n$ with $x_s\in\fc$.
First of all $\overline{J}\subset \overline{V_{(d)}}\subset V_{(\leq d)}$, so $\dim\Oc\leq d$ for any $\Oc\in{\mathcal M}_{\overline J}$. 
We then consider the restriction $
\psi=\varphi|_{\overline J}:\overline J\longrightarrow \varphi(\overline J)
$ to $\overline{J}$
of the Steinberg map $\varphi$. 

Let $z\in \Oc$ for $\Oc\in {\mathcal M}_{\overline J}$. By \cite[Theorem 4]{MR0430168} the irreducible component of the fiber $\psi^{-1}\psi(z)$ containing $z$ is the closure of a $G_0$-orbit in $\overline{J}$, i.e., it is $\overline{\Oc}$. Since $\psi$ is a dominant morphism of irreducible affine varieties, we may argue as in \cite[Corollary 2]{MR0430168} and the fibers of $\psi$ are all of the same dimension, which is the maximum dimension of an orbit in $\overline{J}$, namely $d$. Hence $\dim\Oc=d$, $\Oc\subset \overline{J}^\bullet$ and $$\bigcup_{\Oc\in{\mathcal M}_{\overline J}}\Oc\subset\overline{J}^\bullet\;.$$
The other inclusion follows because $\overline{\Oc}\backslash\Oc$ is always a union of $G_0$-orbits of dimension $<d$.
\end{proof}
\begin{lemma}\label{lem:ss-closure}
Let $J=J_{G_0}(x)$ be a $G_0$-Jordan class and $y=y_s+y_n\in\overline J$. Then:
\begin{itemize}
	\item[$(i)$] $y_s\in \overline J$;
	\item[$(ii)$] $y_s$ is $G_0$-conjugate to an element of $\fzeta(\fg^{x_s})_1$;
	\item[$(iii)$] For any $y'_s\in \fzeta(\fg^{x_s})_1$ there exists a $y'_n\in  \fg^{y'_s}\cap\mathcal{N}_V$ such that $y'_s+y'_n\in\overline{J}^\bullet$.
\item[$(iv)$]  If $z\in\fzeta(\fg)_1$, then $z+y\in \overline{J}$, and in that case $z+y\in\overline{J}^\bullet$ if and only if $y\in\overline{J}^\bullet$.
\end{itemize} 
\end{lemma}
\begin{proof}
Since $\overline{J}$ is $G_0$stable,  claim (i) follows from \cite[Proposition 4]{MR0430168} because $y_s\in\overline{\Oc_y^{G_0}}\subset\overline{J}$.
We now turn to (ii). We may assume $y_s\in\fc$ by Theorem \ref{thm:472}. Claim  (ii) is then an immediate consequence of the following identity
\begin{equation}
\label{eq:JcapC}
\overline J\cap\fc=W_{Vin}\cdot\fzeta(\fg^{x_s})_1\;,
\end{equation}
which we now establish. 

First of all $W_{Vin}\cdot\fzeta(\fg^{x_s})_1\subset \overline J$ by \eqref{eq:usefulII} and $W_{Vin}\cdot\fzeta(\fg^{x_s})_1\subset\fc$ by Corollary \ref{cor:serveanchedopo},
so one inclusion is clear. 
Conversely 
$\varphi(\overline J\cap\fc)\subset\varphi(\overline J)=\varphi(W_{Vin}\cdot\fzeta(\fg^{x_s})_1)$,
by Lemma \ref{lem:image-restricted-Steinberg}.  It follows that
$\overline J\cap\fc\subset W_{Vin}\cdot\fzeta(\fg^{x_s})_1$, since the restriction of $\varphi$ to $\fc$ is just the natural projection to $\fc/W_{Vin}$ and both sets are $W_{Vin}$-stable. 

We prove (iii). By Proposition \ref{eq:JordanIrr} we have that $\fzeta(\fg^{x_s})_1^{reg}+x_n\subset J$, so $\fzeta(\fg^{x_s})_1+x_n\subset\overline{J}$ and $y'_s+x_n\in\overline{J}$. Therefore the orbit $\Oc_{y'_s+x_n}^{G_0}$ is contained in the closure of an orbit $\Oc$ in ${\mathcal M}_{\overline J}$. Since the fibers of the Steinberg map are closed and \cite[Theorem 3]{MR0430168} is in force, $\Oc$ is represented by an element of the form $y'_s+y'_n$ for some $y'_n\in\fg^{y'_s}\cap {\mathcal N}_V$. Clearly $\Oc\subset \overline{J}^\bullet$ by Proposition \ref{prop:maxorbI}.

Finally, (iv) follows from the action of $\fzeta(\fg)_1$ on $J$, cf. Remark \ref{rem:definition-Jordan-classes} \eqref{item:centro}.
\end{proof}
\begin{corollary}
The $G_0$-regular closure $\overline J^\bullet$ of a $G_0$-Jordan class $J$ contains at least a nilpotent $G_0$-orbit.
\end{corollary}

\subsection{A local study of the closure of a $G_0$-Jordan class}
\label{sec:local-study-closure}
We start with a local characterization of the closure of a $G_0$-Jordan class. 
\begin{lemma}\label{lem:local-global}The following statements are equivalent for a $G_0$-Jordan class $J$:
\begin{enumerate}
\item[$(i)$]\label{item:global} $\overline{J}$ is a union of $G_0$-Jordan classes;
\item[$(ii)$]\label{item:local} For every $y\in\overline{J}$ there exists 
a Zariski open neighbourhood  $U_y$ of $y$ in $J_{G_0}(y)$ such that $U_y\subset\overline{J}$.
\end{enumerate}
\end{lemma}
\begin{proof}The implication $(i)\Rightarrow (ii)$ is immediate, since $G_0$-Jordan classes are disjoint and we may take $U_y=J_{G_0}(y)$. Assume now that $(ii)$  holds. Let $y\in\overline{J}$ and set $J'=J_{G_0}(y)$. Then $J'\cap\overline{J}$ is a non-empty closed subset of $J'$. On the other hand, condition $(ii)$ implies that any point of $J'\cap \overline{J}$ has an open neighbourhood of $J'$ therein, therefore $J'\cap\overline{J}$ is also open in $J'$. Since $J'$ is a Zariski irreducible variety, it is connected 
in the Zariski 
topology, 
thus $J'\subset \overline{J}$ and $(i)$ holds. 
\end{proof}
By virtue of Lemma \ref{lem:local-global} we shall apply a local approach and look at the closure of a $G_0$-Jordan class in the neighbourhood of a point of $V$. For the rest of this subsection for any $y_s\in\mathcal S_V$ we will use the following notation: $\;\;\fm:=\fg^{y_s}$;  $M:=G^{y_s}\leq G$; and  $M_0:=M\cap G_0$ with identity component $M_0^\circ$. For any subset $X\subset \fm_1$, we will write $X^{reg,M}$ to indicate the regular part of $X$ for the action of $M$. We also recall that for any GIT quotient $\pi\colon X\to X/\!/H$ of a reductive algebraic group $H$ acting on a variety $X$, a subset $U$ of $X$ is called $\pi$-saturated or $H$-saturated  if $U=\pi^{-1}\pi(U)$. Saturated implies $H$-stable, the converse is not necessarily true.

For $\fm$ as above, we consider the $M_0$-stable subset of $\fm_1$ defined as follows:
\begin{equation*}
U_{\fm}=\{z\in\fm_1|~\fg^{z}\subset\fm\}\;.
\end{equation*}

\begin{lemma}\label{lem:Um}
With notations as above:
\begin{enumerate}
\item[$(i)$] $U_{\fm}$ is $M_0$-saturated;
\item[$(ii)$] $U_{\fm}$ is open in $\fm_1$;
\item[$(iii)$] For all $z=z_s+z_n\in U_{\fm}$ we have $\fzeta(\fg^{z_s})_1^{reg}+z_n=\big(\fzeta(\fm^{z_s})^{reg,M}_1+z_n\big)\cap U_{\fm}$;
\item[$(iv)$] For any $G_0$-Jordan class $J$ such that $J\cap U_{\fm}\neq\emptyset$, we have
\begin{equation}\label{eq:inclusionUm}
J\cap U_{\fm}=\bigcup_{i\in I_{J}}J_{M,i}\cap U_{\fm}\;,
\end{equation}
where 
$\{J_{M,i}~|~i\in I_J\}$ is the (finite) set of $M_0^\circ$-Jordan classes in $\fm_1$ such that $J_{M,i}\cap U_{\fm}\cap J\neq\emptyset$.
In addition, $\dim J_{M,i}= \dim J_{M,j}$ for any $i,j\in I_J$, so the closures of the $J_{M,i}\cap U_{\fm}$'s in \eqref{eq:inclusionUm} are the irreducible components of $J\cap U_{\fm}$;
\item[$(v)$] Let $y_s\in \mathcal S_V$ and $y=y_s+y_n$ for $y_n\in{\mathcal N}_V\cap \fm$. Then 
\begin{equation}\label{eq:normalizer}
J_{G_0}(y)\cap U_{\fm}=\fzeta(\fm)^{reg}_1+\bigcup_{n_i\in N_{G_0}(\fm)/M_0^\circ}n_i\cdot \Oc_{y_n}^{M_0^\circ}
\end{equation} and the locally closed sets  $\fzeta(\fm)^{reg}_1+n_i\cdot \Oc_{y_n}^{M_0^\circ}$ are the irreducible components of $J_{G_0}(y)\cap U_{\fm}$.
\end{enumerate}
\end{lemma}
\begin{proof}For $m = 2$, parts (i)-(ii) are  \cite[Lemma 2.1]{BH}. We propose a slightly different proof for (i). 
Saturation is equivalent to say that
$\fg^{z}\subset\fm$ if and only if $\fg^{z_s}\subset\fm $, for any $z=z_s+z_n\in\fm_1$. 
As $\fg^z=\fg^{z_s}\cap\fg^{z_n}$, one implication is immediate. We will now show that $\fg^{z_s}\not\subset\fg^{s}$  implies $\fg^z\not\subset\fg^{s}$  for any semisimple element $s\in \fg$ and any $z\in \fg^s$, independently of the $\ZZ_m$-grading. Since $z_s$ and $s$ commute, we can always find a  Cartan subalgebra  $\fh$ of $\fg$ containing both. Then
\begin{equation*}
\fg^s=\fh\oplus\left(\bigoplus_{\alpha\in\Phi(s)}\fg_\alpha\right)\;,\qquad\fg^{z_s}=\fh\oplus\left(\bigoplus_{\alpha\in\Phi(z_s)}\fg_\alpha\right)\;,
\end{equation*}
where  $\Phi(h)$ is the set of roots vanishing on an element $h\in\fh$. Since $(\Phi(s)+(\Phi\setminus\Phi(s)))\cap\Phi\subset \Phi\setminus\Phi(s)$, the reductive subalgebra $\fg^s\cap\fg^{z_s}=\fh\oplus \left(\bigoplus_{\alpha\in\Phi(s)\cap\Phi(z_s)}\fg_\alpha\right)$ stabilizes the subspace $X=\bigoplus_{\alpha\in\Phi(z_s)\setminus\Phi(s)}\fg_\alpha$. As $z_n\in \fg^s\cap\fg^{z_s}$ acts nilpotently on $X$,  there is a non-zero $\xi$ in there such that $[z_n,\xi]=0$. In other words $\xi\in\fg^{z}\setminus\fg^s$. 

%

To prove $(ii)$  we use the argument in \cite[Lemma 2.1]{broer}. We may assume $y_s\in \fc$ and that $\fh$ is a Cartan subalgebra of $\fg$, hence of $\fm$, containing $\fc$. The product $f=\prod_{\alpha\in\Phi\setminus\Phi(y_s)}\alpha$ is a homogeneous polynomial on $\fh$ that is invariant for the Weyl group of $\fm$. By Chevalley's restriction theorem $f$ extends to an $M$-invariant polynomial $F$ on $\fm$. By (i), $U_{\fm}=\{z\in\fm_1~|~\fg^{z_s}\subset\fm\}$, and it is not hard to verify that this is
equal to $\{z\in \fm_1~|~F(z)\neq0\}$, hence it is open in $\fm_1$.

Since $U_{\fm}$ is $M_0$-saturated, it is enough to prove $(iii)$ for $z=z_s\in U_{\fm}$. We have $\fg^{z_s}=\fm^{z_s}$, so $\fzeta(\fg^{z_s})_1=\fzeta(\fm^{z_s})_1$.  If $x\in \fzeta(\fg^{z_s})_1^{reg}$ then 
$\fg^x=\fg^{z_s}=\fm^{z_s}\subset\fm$, so $x\in \fzeta(\fm^{z_s})_1^{reg,M}\cap U_{\fm}$. Conversely, if $x\in  \fzeta(\fm^{z_s})_1^{reg,M}\cap U_{\fm}$, then 
$\fg^x\subset\fm$, so $\fg^x=\fm^x=
\fg^{z_s}$ and $x\in \fzeta(\fg^{z_s})_1^{reg}$. 

We prove (iv). Clearly $J\cap U_{\fm}\subset \bigcup_{i\in I_{J}}J_{M,i}\cap U_{\fm}$, and we now show the other inclusion.
Let $z=z_s+z_n\in J\cap J_{M,i}\cap U_{\fm}$ for some $i\in I_{J}$, so $J=J_{G_0}(z)$ and $J_{M,i}=J_{M_0^\circ}(z)$.
Combining the fact that $U_{\fm}$ is $M_0^\circ$-stable  with (iii)  gives
\begin{align*}
J_{M,i}\cap U_{\fm}&=\big(M_0^\circ\cdot(\fzeta(\fm^{z_s})_1^{reg,M}+z_n)\big)\cap U_{\fm}=M_0^\circ\cdot\big((\fzeta(\fm^{z_s})_1^{reg,M}+z_n)\cap U_{\fm}\big)\\
&=M_0^\circ\cdot(\fzeta(\fg^{z_s})_1^{reg}+z_n)\subset G_0\cdot(\fzeta(\fg^{z_s})_1^{reg}+z_n)=J\;,
\end{align*}
establishing \eqref{eq:inclusionUm}. Corollary \ref{cor:JClc} then gives
\begin{align*}\dim J_{M,i}=\dim M_0^\circ-\dim\fm^{z}_0+\dim\fzeta(\fm^{z_s})_1=\dim M_0^\circ -\dim\fg^z_0+\dim\fzeta(\fg^{z_s})_1,\end{align*}
which is independent of $i\in I_J$. Equation \eqref{eq:inclusionUm} expresses $J\cap U_\fm$ as the finite disjoint union 
of irreducible locally closed subsets of the same dimension, and the claim on the irreducible components follows directly. 

Finally, we prove (v). By construction,
\begin{align*}
\fzeta(\fm)^{reg}_1+\bigcup_{n_i\in N_{G_0}(\fm)/M_0^\circ}n_i\cdot \Oc_{y_n}^{M_0^\circ}&= U_{\fm}\cap \big(\fzeta(\fm)^{reg}_1+N_{G_0}(\fm)\cdot y_n\big)\\
&= U_{\fm}\cap \big(N_{G_0}(\fm)\cdot( \fzeta(\fm)^{reg}_1+ y_n)\big)\subset U_{\fm}\cap J_{G_0}(y).
\end{align*}
Conversely, let $z\in J_{G_0}(y)\cap U_{\fm}$. Then, there is $g\in G_0$ such that $\fg^{z_s}=g\cdot\fm$ and $z_n=g\cdot y_n$. Saturation of $U_{\fm}$ gives $\fg^{z_s}\subset \fm$, 
hence $g\cdot\fm\subset\fm$ 
 and $$z\in  N_{G_0}(\fm)\cdot( \fzeta(\fm)^{reg}_1+ y_n)= \fzeta(\fm)^{reg}_1+N_{G_0}(\fm)\cdot y_n= \fzeta(\fm)^{reg}_1+\bigcup_{n_i\in N_{G_0}(\fm)/M_0^\circ}n_i\cdot \Oc_{y_n}^{M_0^\circ}\;,$$
establishing \eqref{eq:normalizer}.  Last claim  follows from (iv) once we prove that the sets $\fzeta(\fm)^{reg}_1+n_i\cdot \Oc_{y_n}^{M_0^\circ}$ are closed in $J_{G_0}(y)\cap U_{\fm}$. 
On the one hand
\begin{align*}
\overline{\fzeta(\fm)^{reg}_1+n_i\cdot \Oc_{y_n}^{M_0^\circ}}^{J_{G_0}(y)\cap U_{\fm}}\subset\overline{\fzeta(\fm)_1+n_i\cdot \Oc_{y_n}^{M_0^\circ}}^{\fm_1}
=\fzeta(\fm)_1+\overline{n_i\cdot \Oc_{y_n}^{M_0^\circ}}^{\fm_1}.
\end{align*}
On the other hand, if $z\in \overline{\fzeta(\fm)^{reg}_1+n_i\cdot \Oc_{y_n}^{M_0^\circ}}^{J_{G_0}(y)\cap U_{\fm}}$, then \eqref{eq:normalizer}  gives  $z\in \fzeta(\fm)^{reg}_1+n_j\cdot \Oc_{y_n}^{M_0^\circ}$ for some $j$.
Hence, $z\in (\fzeta(\fm)_1+n_j\cdot \Oc_{y_n}^{M_0^\circ})\cap (\fzeta(\fm)_1+\overline{n_i\cdot \Oc_{y_n}^{M_0^\circ}}^{\fm_1})$. Since the two nilpotent orbits have the same dimension, we necessarily have $i=j$, completing the proof.
\end{proof}
Let $H$ and $L$ be reductive algebraic groups acting on an affine variety $X$,  with $H\subset L$. Then $H$ acts with trivial stabilizers on the product $L\times X$ via 
$h\cdot(l,x)=(lh^{-1},h\cdot x)$: we set 
$L\times^{H}X:=(L\times X)/{H}\cong(L\times X)/\!/{H}$
and note that $L$ acts on $L\times^{H}X$ by multiplication from the left. 
The class of $(l,x)\in L\times X$ will be denoted by the symbol $l*x\in L\times^{H}X$.  Since $H$ acts on $L\times X$ with trivial stabilizer, the projection $L\times X\to L\times^{H}X$ is a principal $H$-bundle \cite[III.1, Corollaire 1]{luna}.  In other words, there is a surjective \'etale map $Y\to L\times^{H}X$ such that the base change of $L\times X\to L\times^H X$ through $Y\to L\times^{H}X$ is isomorphic to the projection $X\times Y\to Y$. Hence, $A$ is an open  subset of  $L\times^{H}X$ if and only if its pre-image is an open $H$-stable subset of $L\times X$.  In particular, $L\times^H A'$ is open in  $L\times^H X$ for any open $H$-stable subset $A'$ of $X$.  

We consider the natural action maps 
\begin{align}
\label{eq:action-maps}
\tilde{\mu}&\colon G\times\fm\to \fg\;,&&\tilde{\mu}_0\colon G_0\times\fm_1\to V\;,
\end{align} 
and the induced maps $
{\mu}\colon G\times^M\fm\to \fg$ and ${\mu}_0\colon G_0\times^{M_0}\fm_1\to V$. 
\begin{lemma}
\label{lem:etaleandsmooth}
The maps $\tilde\mu, \tilde\mu_0$ are smooth at $(1,y_s)$ and the induced maps $\mu, \mu_0$ are \'etale at $1*y_s$.
\end{lemma}
\begin{proof}
The differential of the map $\tilde{\mu}$ at $(1,y_s)$  maps any element $(x',y')\in \fg\oplus\fm$ to $[x',y_s]+y'$, therefore it is surjective 
by \eqref{eq:orthogonal} and since $[\fg,y_s]\cap\fm=0$ due to the semisimplicity of $y_s$. The differential of the induced map $\mu$ at $1*y_s$ is  also surjective, hence bijective by dimensional reasons. The restriction of the differential of the map $\tilde{\mu}$ at $(1,y_s)$ to the degree $1$ terms readily implies the surjectivity of the differential of  $\tilde{\mu}_0$ at $(1,y_s)$, whence the differential of  $\mu_0$ at $1*y_s$ is bijective.   
\end{proof}

%
We will also consider
the GIT quotient maps 
\begin{align*}
\pi_0\colon  G_0\times^{M_0}\fm_1\to \big(G_0\times^{M_0}\fm_1\big)/\!/G_0\;,&&\pi_{\fm_1}\colon \fm_1\to \fm_1/\!/M_0
\end{align*} 
associated with multiplication from the left by $G_0$,  and the adjoint action of  $M_0$,  respectively.  Recall that  we have natural identifications $\big(G_0\times^{M_0}\fm_1\big)/\!/G_0\cong \fm_1/\!/M_0$.  We will invoke a variant of Luna's \'etale slice Theorem \cite{luna} and its consequences to deduce properties of the closure of $G_0$-Jordan classes. 

\begin{corollary}\label{lem:mu0}
There exists
an affine  open neighbourhood ${\mathcal U}$ of $y_s$ in $\fm_1$, 
which is $M_0$-saturated 
and
such that the restriction of $\mu_0$ to $G_0\times^{M_0}{\mathcal U}$ is \'etale with Zariski open image $G_0\cdot{\mathcal U}$ in $V$.
\end{corollary}
\begin{proof}
By Lemma \ref{lem:etaleandsmooth}, the map $\mu_0$ is \'etale at $1*y_s$. 
The orbit $\Oc_{1*y_s}^{G_0}$ is closed (since it is an orbit of minimal dimension $\dim G_0-\dim M_0$)
and so is the semisimple orbit $\Oc_{y_s}^{G_0}$, cf. \cite[Prop. 3]{MR0430168}. It is also not hard to verify that the restriction of $\mu_0$ to $\Oc_{1*y_s}^{G_0}$ 
is injective. 
By \cite[Lemme Fondamental, \S II.2]{luna} applied to $X=G_0\times^{M_0}\fm_1$ and $Y=V$ there exists  an affine  $\pi_0$-saturated open neighbourhood  $U_X$  of $1*y_s$ in  $G_0\times^{M_0}\fm_1$ such that 
the restriction of $\mu_0$ to it  is \'etale and the image  is an affine open subset of $V$, saturated for 
the Steinberg map $\varphi\colon V\to V/\!/ G_0$. In fact $U_X$ is $G_0\times^{M_0}\mathcal U$ where $\mathcal U=\pi^{-1}_{\fm_1}\pi_0(U_X)$ is the desired $M_0$-saturated affine  open neighbourhood of $y_s$ in $\fm_1$.
\end{proof}

\begin{proposition}\label{prop:inclusion-closure}
Let $J$ be a $G_0$-Jordan class in $V$ and $y=y_s+y_n$, where $y_s\in \mathcal S_V$ and $y_n\in{\mathcal N}_V\cap \fm$.
Then
$y\in \overline J$ if and only if $y\in \overline{J_{M,l}}$ for some $l\in I_{J}$, where
$\{J_{M,i}~|~i\in I_J\}$ is as in Lemma \ref{lem:Um} $(iv)$.
\end{proposition}
\begin{proof}
Consider the restriction of $\mu_0$ from $G_0\times^{M_0}\mathcal U$ to $V$ determined in Corollary \ref{lem:mu0}.
Assume $y\in\overline J$, so $y_s\in\overline J$ by Lemma \ref{lem:ss-closure}.   We consider the commutative diagram \\
$$\begin{CD}G_0\times^{M_0}(\overline{J}\cap\mathcal{U})     @>>>  G_0\cdot\mathcal{U}\cap\overline{J}\\@V{j}VV   @VVV\\
G_0\times^{M_0}\mathcal{U}     @>\mu_0>>  G_0\cdot\mathcal{U}\\\end{CD}$$
 We claim that it is a pull-back diagram. 
 Indeed, the pull-back of $\mu_0$ through the closed inclusion $\overline J\cap G_0\cdot\mathcal{U}\to G_0\cdot\mathcal{U}$ is the restriction $X:=\mu_0^{-1}(\overline{J})\to\overline{J}\cap G_0\cdot\mathcal{U}$. Moreover,  $\overline{J}$ is a reduced subscheme and $\mu_0$ is \'etale, so  $X$ is a $G_0$-stable closed subscheme of  $G_0\times^{M_0}{\mathcal U}$. By \cite[Lemma 3]{luna}, there exists a closed $M_0$-stable subscheme $Y'$ of ${\mathcal U}$ such that $X=G_0\times^{M_0}Y'$ and  the natural morphism $X\to G_0\times^{M_0}\mathcal{U}$ is induced by the inclusion $Y'\subset\mathcal{U}$.  Observe  that $Y'=\mu_0(1*Y')\subset \overline{J}\cap\mathcal{U}$, so the morphism $X\to G_0\times^{M_0}\mathcal{U}$ factors through the morphism $f\colon X\to G_0\times^{M_0}(\overline{J}\cap\mathcal{U})$. Let  $f'\colon G_0\times^{M_0}(\overline{J}\cap\mathcal{U})\to X$ be the unique morphism  coming from the universal property of the pull-back. Now $j\circ f\circ f'=j$, hence $f\circ f' ={\rm id}$  whilst  $f'\circ f={\rm id}$ follows from the universal property.
Hence, the restriction of $\mu_0$ to $G_0\times^{M_0}(\overline{J}\cap\mathcal{U})\to \overline{J}$   is again \'etale, whence open. As $U_{\fm}$ is open in $\fm$, 
$U_\fm\cap\mathcal U\cap\overline J$ is open in $\mathcal U\cap\overline J$.
Thus, $G_0\times^{M_0}(U_\fm\cap\mathcal U\cap\overline J)\to\overline J$ is an open morphism, too.

If $X$ is an irreducible component of $U_{\fm}\cap\mathcal U\cap\overline J$ containing $y$, then $G_0\times^{M_0}X$ dominates the irreducible set $\overline J$, so 
$G_0\cdot X\cap J\neq\emptyset$ because $J$ is locally closed and thus $X\cap J\neq\emptyset$ by  $G_0$-stability. Note also that $X\cap J$ is open in $X$, hence it is irreducible and dense in $X$.
Now 
$$X\cap J\subset J\cap U_{\fm}=\bigcup_{i\in I_{J}}J_{M,i}\cap U_{\fm}$$ by Lemma \ref{lem:Um} (iv),
so $X\cap J\subset \overline{J_{M,l}\cap U_\fm}^{J\cap U_\fm}\subset \overline{J_{M,l}}$ for some $l\in I_J$ by irreducibility. Finally $\overline X=\overline{X\cap J}\subset \overline{J_{M,l}}$ by density and $y\in \overline{J_{M,l}}$ follows since $y\in X$.

Assume now $y\in \overline{J_{M,l}}$ for some $l\in I_{J}$. Then $\overline{J_{M,l}\cap U_\fm}=\overline{J_{M,l}}$ since $J_{M,l}$ is irreducible and $U_\fm$ is open in $\fm_1$, so
$
y\in \overline{J_{M,l}}= \overline{J_{M,l}\cap U_{\fm}}\subset\overline{J\cap U_{\fm}}
\subset
\overline J
$.
\end{proof}

\begin{corollary}\label{cor:mu}
Let $y_s\in \mathcal S_V$ and let $y=y_s+y_n$ for $y_n\in{\mathcal N}_V\cap \fm$. Then there exists
an $M_0^\circ$-stable open neighbourhood $U_1$ of $y$ in $\fm_1$ such that $M_0\cdot U_1\subset  U_{\fm}$ and 
\begin{equation}
\label{eq:intersection}
J_{G_0}(y)\cap U_1=J_{M^\circ_0}(y)\cap U_1.
\end{equation}
\end{corollary}
\begin{proof}
Lemma \ref{lem:Um} (v) gives  
\begin{equation*}
\begin{aligned}
J_{G_0}(y)\cap U_{\fm}&=\big(\fzeta(\fm)^{reg}_1+\bigcup_{n_i\in N_{G_0}(\fm)/M_0^\circ}n_i\cdot \Oc_{y_n}^{M_0^\circ}\big)=\big(\fzeta(\fm)_1+\bigcup_{n_i\in N_{G_0}(\fm)/M_0^\circ}n_i\cdot \Oc_{y_n}^{M_0^\circ}\big)\cap U_{\fm}\;,
\end{aligned}
\end{equation*}
since $\fzeta(\fm)^{reg,M}_1=\fzeta(\fm)_1$. The subsets $(\fzeta(\fm)_1+n_i\cdot \Oc_{y_n}^{M_0^\circ})\cap U_{\fm}$ are the irreducible components of  $J_{G_0}(y)\cap U_{\fm}$ by Lemma \ref{lem:Um} (v).
They are disjoint, whence open,  and  $M_0^\circ$-stable. Hence, there exists an $M_0^\circ$-stable open subset ${\mathcal U}_y$ in $\fm_1$  such that  $J_{G_0}(y)\cap U_{\fm}\cap{\mathcal U}_y=(\fzeta(\fm)_1+\Oc_{y_n}^{M_0^\circ})\cap U_{\fm}$.
Then $U_1:={\mathcal U}_y\cap U_{\fm}$ is an  $M_0^\circ$-stable Zariski open  neighbourhood of $y$ in $\fm_1$ satisfying
$$J_{G_0}(y)\cap U_1=(\fzeta(\fm)_1+\O_{y_n}^{M_0^\circ})\cap U_1=J_{M^\circ_0}(y)\cap U_1.$$ Finally $M_0\cdot U_1\subset M_0\cdot U_{\fm}= U_{\fm}$ since $U_{\fm}$ is $M_0$-stable. \end{proof}
\begin{theorem}\label{prop:local_holds}
Let $J$ be a $G_0$-Jordan class in $V$. Then $\overline{J}$ is a union of $G_0$-Jordan classes and it is decomposable, i.e., it contains the semisimple and nilpotent components of all its elements.
\end{theorem}
\begin{proof}
We will show that condition $(ii)$ in Lemma \ref{lem:local-global} is satisfied for any $y=y_s+y_n\in\overline{J}$. 

Let $U_1$ be as in Corollary \ref{cor:mu} and $\mathcal U$ as in Corollary \ref{lem:mu0}. By saturation of ${\mathcal U}$ we have $y\in U_1\cap {\mathcal U}$.  
We consider the $M_0$-stable open subset $U_1'=M_0\cdot U_1\subset U_{\fm}$ of $\fm_1$ and apply $M_0$ to both sides of \eqref{eq:intersection} to get
\begin{equation}\label{eq:intersection-M0}
J_{G_0}(y)\cap U_1'=M_0\cdot\big(J_{G_0}(y)\cap U_1\big)=M_0\cdot\big(J_{M_0^\circ}(y)\cap U_1\big)\subset \big(M_0\cdot J_{M_0^\circ}(y)\big)\cap U_1'.
\end{equation}
We then set ${\mathcal V} =U'_1\cap{\mathcal U}$ and obtain an $M_0$-stable open neighbourhood of $y$ in
$\fm_1$.  Corollary \ref{lem:mu0} then guarantees that $G_0\cdot {\mathcal V}$ is an open neighbourhood of $y$ in $V$. We will show that $U_y:=J_{G_0}(y)\cap G_0\cdot {\mathcal V}$ is the sought neighbourhood of $y$ in $J_{G_0}(y)$ contained in $\overline{J}$. 

Proposition \ref{prop:inclusion-closure} ensures that $y\in \overline{J_{M, l}}\cap \mathcal V$ for  some $l\in I_{J}$. Now $y_s\in\fzeta(\fm)_1$, so
combining \eqref{eq:intersection-M0}, Remark \ref{rem:definition-Jordan-classes} \eqref{item:centro} and Lemma   \ref{lem:ss-closure} (iv) we get 
\begin{equation*}
\begin{aligned}
J_{G_0}(y)\cap {\mathcal V}&\subset \big(M_0\cdot J_{M_0^\circ}(y)\big)\cap{\mathcal V}
=\Big(M_0\cdot\big(\fzeta(\fm)_1+\O_{y_n}^{M_0^\circ}\big)\Big)\cap {\mathcal V}\\
&\subset  \Big(M_0\cdot\overline{J_{M, l}}\Big)\cap\mathcal V
=M_0\cdot \Big(\overline{J_{M, l}}\cap \mathcal V\Big)
\subset \overline{J}\cap \mathcal V\;,
\end{aligned}
\end{equation*}
where for the last inclusion we invoke Proposition \ref{prop:inclusion-closure} once more. 
We finally arrive at
\begin{align*}
J_{G_0}(y)\cap G_0\cdot \mathcal V = 
G_0 \cdot (J_{G_0}(y) \cap \mathcal{V})\subset
G_0\cdot (\overline{J}\cap \mathcal V)\subset \overline{J},\end{align*}
This proves that $\overline{J}$ is the union of $G_0$-Jordan classes.

We finally prove that $\overline{J}$ is decomposable. Let $y=y_s+y_n\in\overline J$ and $J_{G_0}(y)$ the corresponding $G_0$-Jordan class. Then 
$y_s\in\overline J$ by Lemma \ref{lem:ss-closure} (i) and $$y_n\in \fzeta(\fm)_1+y_n= \overline{\fzeta(\fm)_1^{reg}+y_n}\subset \overline{J_{G_0}(y)}\subset \overline J\;,$$
where we used our previous result $J_{G_0}(y)\subset\overline J$.
\end{proof}

\begin{theorem}
\label{thm:union-Jordan-classes}
Let $J$ be a  $G_0$-Jordan class and let $S$ be a sheet in $V$. Then $\overline{J}^\bullet$, $\overline{J}^{reg}$ and $S$ are unions of $G_0$-Jordan classes.
\end{theorem}
\begin{proof}By Theorem \ref{prop:local_holds}, the closure $\overline{J}$ is a union of $G_0$-Jordan classes. Since all such classes are 
of constant $G$- and $G_0$-orbit dimension,  
it follows that also $\overline{J}^\bullet$ and $\overline{J}^{reg}$ are unions of $G_0$-Jordan classes. The statement for $S$  is a direct consequence of Proposition \ref{prop:dense}.
\end{proof}
We conclude this subsection with the following important consequence of the local study of the closure of a $G_0$-Jordan class.
\begin{proposition}
\label{prop:Jordan-classes-are-smooth}
$G_0$-Jordan classes are smooth.
\end{proposition}
\begin{proof}
Let $J=J_{G_0}(y)$ be a $G_0$-Jordan class in $V$ and $\fm=\fg^{y_s}$.  We will show that $y$ has a smooth Zariski open neighbourhood in $J$.  
Let $U_{\fm}$ and ${\mathcal U}$ be the  open neighbourhoods of $y_s$ and $y$ in $\fm_1$
as in Lemma \ref{lem:Um} and Corollary \ref{lem:mu0}, respectively. 

By Lemma \ref{lem:Um} (v), the intersection $J\cap U_{\fm}$  is smooth, therefore $J\cap U_{\fm}\cap {\mathcal U}$ is non-empty and smooth as well.
Recall that 
$p\colon G_0\times\overline{J}\to G_0\times^{M_0}\overline{J}$ is a principal $M_0$-bundle, so there is a surjective \'etale map $f\colon Y\to G_0\times^{M_0}\overline{J}$ such that the base change $X\to Y$ of $G_0\times\overline{J}\to G_0\times^{M_0}\overline{J}$ is isomorphic to the projection $\tilde{p}\colon M_0\times Y\to Y$. Being the base change of an \'etale 
map, the induced morphism $\tilde{f}\colon M_0\times Y\to G_0\times\overline{J}$  is again so. By \cite[\'Exp 1, Corollaire 9.2]{SGA1}, $G_0\times(J\cap U_{\fm}\cap{\mathcal U})$ is smooth if and only if $\tilde{f}\tilde{p}^{-1}f^{-1}(G_0\times^{M_0} (J\cap U_{\fm}\cap{\mathcal U}))=p^{-1} \big(G_0\times^{M_0}(J\cap U_{\fm}\cap{\mathcal U})\big)$ is so. One may verify that the scheme-theoretic fiber of $G_0\times^{M_0} (J\cap U_{\fm}\cap{\mathcal U})$ through $p$ is $G_0\times(J\cap U_{\fm}\cap{\mathcal U})$ hence $G_0\times^{M_0} (J\cap U_{\fm}\cap{\mathcal U})$ is smooth.
Invoking again  \cite[\'Exp 1, Corollaire 9.2]{SGA1} we conclude that  $\mu_0\big(G_0\times^{M_0} (J\cap U_{\fm}\cap{\mathcal U})\big)$ is smooth and it is  a smooth open neighbourhood of $y$ in $J$.
\end{proof}

\subsection{Regularity questions}
\label{subsec:regularity questions}
Let $J=J_{G_0}(x_s+x_n)$ be a $G_0$-Jordan class. Then $\overline{J}^{reg}\subset\overline{J}^\bullet$ since $J$ is irreducible, hence $\overline J$ too, and Lemma \ref{lemma:regvsbullet} is in force. Note that $\overline{J}^\bullet=\overline{J}^{reg}$ whenever $x_s=0$, because 
$J=\fzeta(\fg)_1\times\O^{G_0}_{x_n}$ 
and orbits are locally closed, so $J=\overline{J}^\bullet=\overline{J}^{reg}$.
The equality $\overline{J}^\bullet=\overline{J}^{reg}$ is always satisfied in the symmetric case $m=2$ due to Corollary \ref{corollaryK-R}
and one may wonder if $\overline{J}^\bullet=\overline{J}^{reg}$ also for $m\geq 3$, by
combining Theorem \ref{thm:union-Jordan-classes} and the fact that $G_0$-Jordan classes are defined in terms of 
regular parts for the action of $G$, cf. 
Corollary \ref{cor:JClc}.

However, this is not the case. A reason is that open $G_0$-orbits $\O^{G_0}$ in irreducible components of the fibers of the Steinberg map $\varphi\colon V\to V/\!/ G_0\cong \fc/W_{Vin}$ do {\it not} give rise in general to open $G$-orbits $G\cdot \O^{G_0}$ in the irreducible components of the Steinberg map $p\colon\fg\to\fg/\!/G\cong \fh/W$.
To make this more precise, we need some notions and results from \cite{MR2111215, MR2504930} and, for simplicity of exposition, we restrict to the case where $\fg$ is semisimple.
\begin{definition} A complex semisimple $\ZZ_m$-graded Lie algebra $\{\fg,\theta,m\}$, the corresponding grading, and the automorphism $\theta$ are called:
\begin{itemize}
\item[(i)] {\it $\mathcal S$-regular} if $\mathcal S_V\cap\fg^{reg}\neq \emptyset$;	
\item[(ii)] {\it $\mathcal N$-regular} if $\mathcal N_V\cap\fg^{reg}\neq \emptyset$;	
\item[(iii)] {\it very $\mathcal N$-regular} if each irreducible component of $\mathcal N_V$ intersects $\fg^{reg}$ non-trivially.
\end{itemize}
\end{definition}
Clearly (iii) implies (ii). 
It is an important result of L.\ V.\ Antonyan and D.\ I.\ Panyushev in \cite{MR2111215}  that if a connected component of $\operatorname{Aut}(\fg)$ contains automorphisms
of order $m$, then it contains a unique {\it $\mathcal N$-regular} automorphism 
of that order (up to conjugation by the group of inner automorphisms of $\fg$).
Moreover, as mentioned in the introduction of \cite{MR2111215}, the condition of $\mathcal S$-regularity is equivalent to $\mathcal N$-regularity in the symmetric case 
$m=2$, but for $m\geq 3$ neither of these
properties implies the other. An example of $\mathcal S$-regular grading that is not $\mathcal N$-regular is given in \cite[Example 4.5]{MR2111215}. Here $\fg$ is of type $E_6$ with the inner automorphism of order $m=4$ described by the Kac diagram 
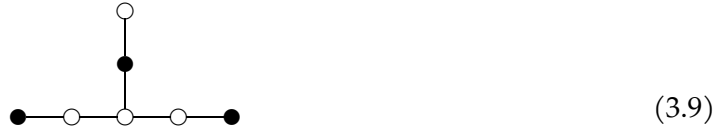
\begin{equation}
\label{eq:exampleE6}
\begin{picture}(90,50)
  \put(3,0){\circle*{6}}
  \put(23,0){\circle{6}}
  \put(43,0){\circle{6}}
  \put(63,0){\circle{6}}
  \put(83,0){\circle*{6}}
  \put(43,20){\circle*{6}}
  \put(43,40){\circle{6}}
  \put(6,0){\line(1,0){14}}
  \put(26,0){\line(1,0){14}}
  \put(46,0){\line(1,0){14}}
  \put(66,0){\line(1,0){14}}
  \put(43,3){\line(0,1){14}}
  \put(43,23){\line(0,1){14}}
\end{picture}
\end{equation}
This is the affine Dynkin diagram of $\fg$ of type $E_6$, where the white and black nodes correspond to roots subspaces of degree $0$ and $1$, respectively.
The semisimple part of $\fg_0$ is given by the subdiagram consisting
of white nodes and the dimension of the centre of $\fg_0$ is the number of black nodes
minus $1$. We have $G_0\cong SL(4)\times SL(2)\times  (\mathbb C^x)^2$ up to
local isomorphism, acting on $V=\fg_1\cong \mathbb C^4\oplus(\mathbb C^4)^*\oplus(\Lambda^2 \mathbb C^4\boxtimes \mathbb C^2)$. The reader is referred to e.g. \cite[Chapter 3, \S 3]{MR1349140} for a detailed treatment of periodic automorphisms and their associated Kac diagrams. 

Now $G_0$-Jordan classes form a finite partition of $V$, which is irreducible, so there is one class $J$ that is open in $V$.
We call it the $G_0$-regular Jordan class of $V$ and note that it is the unique $G_0$-Jordan class that is dense in the $G_0$-regular sheet $S=V^\bullet$ of $V$. (See Example \ref{ex:elements-general-position} for an explicit description of representatives of the $G_0$-orbits in the $G_0$-regular Jordan class.)
Since the grading \eqref{eq:exampleE6} is $\mathcal S$-regular, we have $\overline{J}^{reg}=V^{reg}=\fg^{reg}\cap V$ in this case.
Let $\O^{G_0}$ be the nilpotent $G_0$-orbit that is open in one of the irreducible components of $\mathcal N_V$.
We have $\O^{G_0}\subset \overline{J}^{\bullet}=V^\bullet$ by 
\cite[Corollaries 1 and 2]{MR0430168}, but $\O^{G_0}\not\subset\overline{J}^{reg}$ since the grading is not $\mathcal N$-regular.

The cone $\mathcal N_V$ is often reducible and a larger class of examples for which $\overline{J}^{reg}\neq\overline{J}^\bullet$ comes from $\mathcal N$-regular gradings that are not very $\mathcal N$-regular:
the $G_0$-regular Jordan class $J$ satisfies $\overline{J}^{reg}=\fg^{reg}\cap V$ and, by an argument as above,
there is a nilpotent $G_0$-orbit contained in $\overline{J}^{\bullet}$ but not in
$\overline{J}^{reg}$. 
Exceptional $\mathcal N$-regular gradings whose nodes are not all black are classified in \cite{MR2765379},
and very $\mathcal N$-regular gradings appear to occur very rarely. Inner exceptional gradings 
with all nodes black are $\mathcal N$-regular but not very $\mathcal N$-regular \cite[Example 4.4]{MR2111215}
and the same is true for the outer grading of $E_6$ with all nodes black (W. A. de Graaf,  05-05-2020, personal communication).
The following result is a consequence of these observations,
and the tables are a specialization of Tables $2$-$7$ of \cite{MR2765379}.
\begin{proposition}\label{prop:chiusura-reg-bullet}
Let $\{\fg,\theta,m\}$ be an exceptional  complex simple $\ZZ_m$-graded Lie algebra, $m\geq 3$. Then $\{\fg,\theta,m\}$
is $\mathcal N$-regular but not very $\mathcal N$-regular if and only if the associated Kac diagram has all the nodes black
or is one in the following tables. In all these cases we have that $\overline{J}^{reg}\subset \overline{J}^{\bullet}$ properly, where 
 $J$ is the $G_0$-regular Jordan class of $V$. 

{\small
\begin{longtable}{|c|c|c|c|c|c|}
\caption{$\mathcal N$-regular but not very $\mathcal N$-regular automorphisms of $G_2$.}
\endfirsthead
\hline
\multicolumn{6}{|l|}{\small\slshape $\mathcal N$-regular but not very $\mathcal N$-regular automorphisms of $G_2$.} \\
\hline 
\endhead
\hline 
\endfoot
\endlastfoot

\hline

$m$ & Kac diagram & \# orbits in $\mathcal N_V$ & \# components of $\mathcal N_V$ & $\dim\mathcal N_V$ & $\operatorname{dim}\fc$ \\
\hline

3 & 
\begin{picture}(70,7)
  \put(5,0){\circle*{6}}
  \put(35,0){\circle*{6}}
  \put(65,0){\circle{6}}
  \put(8,0){\line(1,0){24}}
  \put(35,-3){\line(1,0){30}}
  \put(38,0){\line(1,0){24}}
  \put(35,3){\line(1,0){30}}
  \put(45,-4){\Large $>$}
\end{picture}
& 6 & 2 & 4 & 1\\
\hline
\end{longtable}

\begin{longtable}{|c|c|c|c|c|c|}
\caption{$\mathcal N$-regular but not very $\mathcal N$-regular automorphisms of $F_4$.}
\endfirsthead
\hline
\multicolumn{6}{|l|}{\small\slshape $\mathcal N$-regular but not very $\mathcal N$-regular inner automorphisms of $F_4$.} \\
\hline 
\endhead
\hline 
\endfoot
\endlastfoot

\hline

$m$ & Kac diagram & \# orbits in $\mathcal N_V$ & \# components of $\mathcal N_V$ & $\dim\mathcal N_V$ & $\operatorname{dim}\fc$ \\
\hline

4 & 
\begin{picture}(130,15)
  \put(10,5){\circle*{6}}
  \put(40,5){\circle{6}}
  \put(70,5){\circle*{6}}
  \put(100,5){\circle{6}}
  \put(130,5){\circle{6}}
  \put(13,5){\line(1,0){24}}
  \put(43,5){\line(1,0){24}}
  \put(72,7){\line(1,0){26}}
  \put(72,3){\line(1,0){26}}
  \put(80,1){\Large $>$}
  \put(103,5){\line(1,0){24}}
\end{picture}
& 29 & 3 & 12 & 2 \\

6 & 
\begin{picture}(130,15)
  \put(10,5){\circle*{6}}
  \put(40,5){\circle{6}}
  \put(70,5){\circle*{6}}
  \put(100,5){\circle{6}}
  \put(130,5){\circle*{6}}
  \put(13,5){\line(1,0){24}}
  \put(43,5){\line(1,0){24}}
  \put(72,7){\line(1,0){26}}
  \put(72,3){\line(1,0){26}}
  \put(80,1){\Large $>$}
  \put(103,5){\line(1,0){24}}
\end{picture}
& 35 & 6 & 8 & 2\\

8 & 
\begin{picture}(130,15)
  \put(10,5){\circle*{6}}
  \put(40,5){\circle*{6}}
  \put(70,5){\circle*{6}}
  \put(100,5){\circle{6}}
  \put(130,5){\circle*{6}}
  \put(13,5){\line(1,0){24}}
  \put(43,5){\line(1,0){24}}
  \put(72,7){\line(1,0){26}}
  \put(72,3){\line(1,0){26}}
  \put(80,1){\Large $>$}
  \put(103,5){\line(1,0){24}}
\end{picture}
& 30 & 4 & 6 & 1\\

\hline
\end{longtable}

\begin{longtable}{|c|c|c|c|c|c|}
\caption{$\mathcal N$-regular but not very $\mathcal N$-regular inner automorphisms of $E_6$.}
\endfirsthead
\hline
\multicolumn{6}{|l|}{\small\slshape $\mathcal N$-regular but not very $\mathcal N$-regular inner automorphisms of $E_6$.} \\
\hline 
\endhead
\hline 
\endfoot
\endlastfoot

\hline

$m$ & Kac diagram & \# orbits in $\mathcal N_V$ & \# components of $\mathcal N_V$ & $\dim\mathcal N_V$ & $\operatorname{dim}\fc$ \\
\hline

4 & 
\begin{picture}(90,50)
  \put(3,0){\circle*{6}}
  \put(23,0){\circle{6}}
  \put(43,0){\circle*{6}}
  \put(63,0){\circle{6}}
  \put(83,0){\circle{6}}
  \put(43,20){\circle{6}}
  \put(43,40){\circle{6}}
  \put(6,0){\line(1,0){14}}
  \put(26,0){\line(1,0){14}}
  \put(46,0){\line(1,0){14}}
  \put(66,0){\line(1,0){14}}
  \put(43,3){\line(0,1){14}}
  \put(43,23){\line(0,1){14}}
\end{picture}
& 43 & 3 & 18 & 2\\

6 & 
\begin{picture}(90,50)
  \put(3,0){\circle*{6}}
  \put(23,0){\circle{6}}
  \put(43,0){\circle*{6}}
  \put(63,0){\circle{6}}
  \put(83,0){\circle*{6}}
  \put(43,20){\circle{6}}
  \put(43,40){\circle*{6}}
  \put(6,0){\line(1,0){14}}
  \put(26,0){\line(1,0){14}}
  \put(46,0){\line(1,0){14}}
  \put(66,0){\line(1,0){14}}
  \put(43,3){\line(0,1){14}}
  \put(43,23){\line(0,1){14}}
\end{picture}
& 133 & 9 & 12 & 2

\\

8 & 
\begin{picture}(90,50)
  \put(3,0){\circle*{6}}
  \put(23,0){\circle{6}}
  \put(43,0){\circle*{6}}
  \put(63,0){\circle*{6}}
  \put(83,0){\circle*{6}}
  \put(43,20){\circle{6}}
  \put(43,40){\circle*{6}}
  \put(6,0){\line(1,0){14}}
  \put(26,0){\line(1,0){14}}
  \put(46,0){\line(1,0){14}}
  \put(66,0){\line(1,0){14}}
  \put(43,3){\line(0,1){14}}
  \put(43,23){\line(0,1){14}}
\end{picture}
& 70 & 4 & 9 & 1\\

9 & 
\begin{picture}(90,50)
  \put(3,0){\circle*{6}}
  \put(23,0){\circle*{6}}
  \put(43,0){\circle{6}}
  \put(63,0){\circle*{6}}
  \put(83,0){\circle*{6}}
  \put(43,20){\circle*{6}}
  \put(43,40){\circle*{6}}
  \put(6,0){\line(1,0){14}}
  \put(26,0){\line(1,0){14}}
  \put(46,0){\line(1,0){14}}
  \put(66,0){\line(1,0){14}}
  \put(43,3){\line(0,1){14}}
  \put(43,23){\line(0,1){14}}
\end{picture}
& 118 & 6 & 8 & 1

\\

\hline
\end{longtable}

\begin{longtable}{|c|c|c|c|c|c|}
\caption{$\mathcal N$-regular but not very $\mathcal N$-regular outer automorphisms of $E_6$.}
\endfirsthead
\hline
\multicolumn{6}{|l|}{\small\slshape $\mathcal N$-regular but not very $\mathcal N$-regular outer automorphisms of $E_6$.} \\
\hline 
\endhead
\hline 
\endfoot
\endlastfoot

\hline

$m$ & Kac diagram & \# orbits in $\mathcal N_V$ & \# components of $\mathcal N_V$ & $\dim\mathcal N_V$ & $\operatorname{dim}\fc$ \\
\hline

6 & 
\begin{picture}(130,15)
  \put(10,5){\circle*{6}}
  \put(40,5){\circle{6}}
  \put(70,5){\circle{6}}
  \put(100,5){\circle*{6}}
  \put(130,5){\circle{6}}
  \put(13,5){\line(1,0){24}}
  \put(43,5){\line(1,0){24}}
  \put(72,7){\line(1,0){26}}
  \put(72,3){\line(1,0){26}}
  \put(80,1){\Large $<$}
  \put(103,5){\line(1,0){24}}
\end{picture}
& 34 & 5 & 12 & 3 \\

8 & 
\begin{picture}(130,15)
  \put(10,5){\circle*{6}}
  \put(40,5){\circle{6}}
  \put(70,5){\circle{6}}
  \put(100,5){\circle*{6}}
  \put(130,5){\circle*{6}}
  \put(13,5){\line(1,0){24}}
  \put(43,5){\line(1,0){24}}
  \put(72,7){\line(1,0){26}}
  \put(72,3){\line(1,0){26}}
  \put(80,1){\Large $<$}
  \put(103,5){\line(1,0){24}}
\end{picture}
& 22 & 3 & 9 & 1\\

10 & 
\begin{picture}(130,15)
  \put(10,5){\circle*{6}}
  \put(40,5){\circle*{6}}
  \put(70,5){\circle{6}}
  \put(100,5){\circle*{6}}
  \put(130,5){\circle{6}}
  \put(13,5){\line(1,0){24}}
  \put(43,5){\line(1,0){24}}
  \put(72,7){\line(1,0){26}}
  \put(72,3){\line(1,0){26}}
  \put(80,1){\Large $<$}
  \put(103,5){\line(1,0){24}}
\end{picture}
& 25 & 2 & 8 & 1\\

12 & 
\begin{picture}(130,15)
  \put(10,5){\circle*{6}}
  \put(40,5){\circle*{6}}
  \put(70,5){\circle{6}}
  \put(100,5){\circle*{6}}
  \put(130,5){\circle*{6}}
  \put(13,5){\line(1,0){24}}
  \put(43,5){\line(1,0){24}}
  \put(72,7){\line(1,0){26}}
  \put(72,3){\line(1,0){26}}
  \put(80,1){\Large $<$}
  \put(103,5){\line(1,0){24}}
\end{picture}
& 30 & 4 & 6 & 1\\

\hline
\end{longtable}

\begin{longtable}{|c|c|c|c|c|c|}
\caption{$\mathcal N$-regular but not very $\mathcal N$-regular automorphisms of $E_7$.}
\endfirsthead
\hline
\multicolumn{6}{|l|}{\small\slshape $\mathcal N$-regular but not very $\mathcal N$-regular automorphisms of $E_7$.} \\
\hline 
\endhead
\hline 
\endfoot
\endlastfoot

\hline

$m$ & Kac diagram & \# orbits in $\mathcal N_V$ & \# components of $\mathcal N_V$ & $\dim\mathcal N_V$ & $\operatorname{dim}\fc$ \\
\hline

6 &
\begin{picture}(140,40)
  \put(3,0){\circle*{6}}
  \put(23,0){\circle{6}}
  \put(43,0){\circle{6}}
  \put(63,0){\circle*{6}}
  \put(83,0){\circle{6}}
  \put(63,20){\circle{6}}
  \put(6,0){\line(1,0){14}}
  \put(26,0){\line(1,0){14}}
  \put(46,0){\line(1,0){14}}
  \put(66,0){\line(1,0){14}}
  \put(63,3){\line(0,1){14}}
\put(103,0){\circle{6}}
\put(86,0){\line(1,0){14}}
\put(123,0){\circle*{6}}
\put(106,0){\line(1,0){14}}
\end{picture} 
& 233 & 10 & 21 & 3\\

7 &
\begin{picture}(140,40)
  \put(3,0){\circle*{6}}
  \put(23,0){\circle{6}}
  \put(43,0){\circle{6}}
  \put(63,0){\circle*{6}}
  \put(83,0){\circle{6}}
  \put(63,20){\circle{6}}
  \put(6,0){\line(1,0){14}}
  \put(26,0){\line(1,0){14}}
  \put(46,0){\line(1,0){14}}
  \put(66,0){\line(1,0){14}}
  \put(63,3){\line(0,1){14}}
\put(103,0){\circle*{6}}
\put(86,0){\line(1,0){14}}
\put(123,0){\circle{6}}
\put(106,0){\line(1,0){14}}
\end{picture} 
& 112 & 3 & 18 & 1\\

8 &
\begin{picture}(140,40)
  \put(3,0){\circle*{6}}
  \put(23,0){\circle{6}}
  \put(43,0){\circle{6}}
  \put(63,0){\circle*{6}}
  \put(83,0){\circle{6}}
  \put(63,20){\circle{6}}
  \put(6,0){\line(1,0){14}}
  \put(26,0){\line(1,0){14}}
  \put(46,0){\line(1,0){14}}
  \put(66,0){\line(1,0){14}}
  \put(63,3){\line(0,1){14}}
\put(103,0){\circle*{6}}
\put(86,0){\line(1,0){14}}
\put(123,0){\circle*{6}}
\put(106,0){\line(1,0){14}}
\end{picture} 
& 163 & 2 & 17 & 1\\

9 &
\begin{picture}(140,40)
  \put(3,0){\circle*{6}}
  \put(23,0){\circle*{6}}
  \put(43,0){\circle{6}}
  \put(63,0){\circle*{6}}
  \put(83,0){\circle{6}}
  \put(63,20){\circle{6}}
  \put(6,0){\line(1,0){14}}
  \put(26,0){\line(1,0){14}}
  \put(46,0){\line(1,0){14}}
  \put(66,0){\line(1,0){14}}
  \put(63,3){\line(0,1){14}}
\put(103,0){\circle*{6}}
\put(86,0){\line(1,0){14}}
\put(123,0){\circle{6}}
\put(106,0){\line(1,0){14}}
\end{picture}
& 132 & 4 & 14 & 1
\\

10 &
\begin{picture}(140,40)
  \put(3,0){\circle*{6}}
  \put(23,0){\circle{6}}
  \put(43,0){\circle*{6}}
  \put(63,0){\circle{6}}
  \put(83,0){\circle*{6}}
  \put(63,20){\circle*{6}}
  \put(6,0){\line(1,0){14}}
  \put(26,0){\line(1,0){14}}
  \put(46,0){\line(1,0){14}}
  \put(66,0){\line(1,0){14}}
  \put(63,3){\line(0,1){14}}
\put(103,0){\circle{6}}
\put(86,0){\line(1,0){14}}
\put(123,0){\circle*{6}}
\put(106,0){\line(1,0){14}}
\end{picture} 
& 199 & 4 & 13 & 1\\

12 &
\begin{picture}(140,40)
  \put(3,0){\circle*{6}}
  \put(23,0){\circle{6}}
  \put(43,0){\circle*{6}}
  \put(63,0){\circle{6}}
  \put(83,0){\circle*{6}}
  \put(63,20){\circle*{6}}
  \put(6,0){\line(1,0){14}}
  \put(26,0){\line(1,0){14}}
  \put(46,0){\line(1,0){14}}
  \put(66,0){\line(1,0){14}}
  \put(63,3){\line(0,1){14}}
\put(103,0){\circle*{6}}
\put(86,0){\line(1,0){14}}
\put(123,0){\circle*{6}}
\put(106,0){\line(1,0){14}}
\end{picture} 
& 217 & 5 & 11 & 1\\

14 &
\begin{picture}(140,40)
  \put(3,0){\circle*{6}}
  \put(23,0){\circle*{6}}
  \put(43,0){\circle*{6}}
  \put(63,0){\circle{6}}
  \put(83,0){\circle*{6}}
  \put(63,20){\circle*{6}}
  \put(6,0){\line(1,0){14}}
  \put(26,0){\line(1,0){14}}
  \put(46,0){\line(1,0){14}}
  \put(66,0){\line(1,0){14}}
  \put(63,3){\line(0,1){14}}
\put(103,0){\circle*{6}}
\put(86,0){\line(1,0){14}}
\put(123,0){\circle*{6}}
\put(106,0){\line(1,0){14}}
\end{picture} 
& 238 & 7 & 9 & 1\\

\hline
\end{longtable}

\begin{longtable}{|c|c|c|c|c|c|}
\caption{$\mathcal N$-regular but not very $\mathcal N$-regular automorphisms of $E_8$.}
\endfirsthead
\hline
\multicolumn{6}{|l|}{\small\slshape $\mathcal N$-regular but not very $\mathcal N$-regular automorphisms of $E_8$.} \\
\hline 
\endhead
\hline 
\endfoot
\endlastfoot

\hline

$m$ & Kac diagram & \# orbits in $\mathcal N_V$ & \# components of $\mathcal N_V$ & $\dim\mathcal N_V$ & $\operatorname{dim}\fc$ \\
\hline

4 &
\begin{picture}(150,40)
  \put(3,0){\circle{6}}
  \put(23,0){\circle{6}}
  \put(43,0){\circle{6}}
  \put(63,0){\circle{6}}
  \put(83,0){\circle*{6}}
  \put(43,20){\circle{6}}
  \put(6,0){\line(1,0){14}}
  \put(26,0){\line(1,0){14}}
  \put(46,0){\line(1,0){14}}
  \put(66,0){\line(1,0){14}}
  \put(43,3){\line(0,1){14}}
\put(103,0){\circle{6}}
\put(86,0){\line(1,0){14}}
\put(123,0){\circle{6}}
\put(106,0){\line(1,0){14}}
\put(143,0){\circle{6}}
\put(126,0){\line(1,0){14}}
\end{picture} 
& 144 & 2 & 60 & 4\\

6 &
\begin{picture}(150,40)
  \put(3,0){\circle{6}}
  \put(23,0){\circle{6}}
  \put(43,0){\circle{6}}
  \put(63,0){\circle*{6}}
  \put(83,0){\circle{6}}
  \put(43,20){\circle{6}}
  \put(6,0){\line(1,0){14}}
  \put(26,0){\line(1,0){14}}
  \put(46,0){\line(1,0){14}}
  \put(66,0){\line(1,0){14}}
  \put(43,3){\line(0,1){14}}
\put(103,0){\circle{6}}
\put(86,0){\line(1,0){14}}
\put(123,0){\circle{6}}
\put(106,0){\line(1,0){14}}
\put(143,0){\circle*{6}}
\put(126,0){\line(1,0){14}}
\end{picture} 
& 270 & 7 & 40 & 4\\

8 &
\begin{picture}(150,40)
  \put(3,0){\circle{6}}
  \put(23,0){\circle{6}}
  \put(43,0){\circle*{6}}
  \put(63,0){\circle{6}}
  \put(83,0){\circle{6}}
  \put(43,20){\circle{6}}
  \put(6,0){\line(1,0){14}}
  \put(26,0){\line(1,0){14}}
  \put(46,0){\line(1,0){14}}
  \put(66,0){\line(1,0){14}}
  \put(43,3){\line(0,1){14}}
\put(103,0){\circle{6}}
\put(86,0){\line(1,0){14}}
\put(123,0){\circle*{6}}
\put(106,0){\line(1,0){14}}
\put(143,0){\circle{6}}
\put(126,0){\line(1,0){14}}
\end{picture} 
& 219 & 2 & 30 & 2\\

9 &
\begin{picture}(150,40)
  \put(3,0){\circle{6}}
  \put(23,0){\circle{6}}
  \put(43,0){\circle*{6}}
  \put(63,0){\circle{6}}
  \put(83,0){\circle{6}}
  \put(43,20){\circle{6}}
  \put(6,0){\line(1,0){14}}
  \put(26,0){\line(1,0){14}}
  \put(46,0){\line(1,0){14}}
  \put(66,0){\line(1,0){14}}
  \put(43,3){\line(0,1){14}}
\put(103,0){\circle{6}}
\put(86,0){\line(1,0){14}}
\put(123,0){\circle*{6}}
\put(106,0){\line(1,0){14}}
\put(143,0){\circle*{6}}
\put(126,0){\line(1,0){14}}
\end{picture} 
& 206 & 2 & 28 & 1\\

10 &
\begin{picture}(150,40)
  \put(3,0){\circle{6}}
  \put(23,0){\circle{6}}
  \put(43,0){\circle*{6}}
  \put(63,0){\circle{6}}
  \put(83,0){\circle{6}}
  \put(43,20){\circle{6}}
  \put(6,0){\line(1,0){14}}
  \put(26,0){\line(1,0){14}}
  \put(46,0){\line(1,0){14}}
  \put(66,0){\line(1,0){14}}
  \put(43,3){\line(0,1){14}}
\put(103,0){\circle*{6}}
\put(86,0){\line(1,0){14}}
\put(123,0){\circle{6}}
\put(106,0){\line(1,0){14}}
\put(143,0){\circle*{6}}
\put(126,0){\line(1,0){14}}
\end{picture} 
& 300 & 7 & 24 & 2\\

12 &
\begin{picture}(150,40)
  \put(3,0){\circle*{6}}
  \put(23,0){\circle{6}}
  \put(43,0){\circle*{6}}
  \put(63,0){\circle{6}}
  \put(83,0){\circle{6}}
  \put(43,20){\circle{6}}
  \put(6,0){\line(1,0){14}}
  \put(26,0){\line(1,0){14}}
  \put(46,0){\line(1,0){14}}
  \put(66,0){\line(1,0){14}}
  \put(43,3){\line(0,1){14}}
\put(103,0){\circle*{6}}
\put(86,0){\line(1,0){14}}
\put(123,0){\circle{6}}
\put(106,0){\line(1,0){14}}
\put(143,0){\circle*{6}}
\put(126,0){\line(1,0){14}}
\end{picture} 
& 398 & 10 & 20 & 2\\

14 &
\begin{picture}(150,40)
  \put(3,0){\circle*{6}}
  \put(23,0){\circle{6}}
  \put(43,0){\circle*{6}}
  \put(63,0){\circle{6}}
  \put(83,0){\circle{6}}
  \put(43,20){\circle{6}}
  \put(6,0){\line(1,0){14}}
  \put(26,0){\line(1,0){14}}
  \put(46,0){\line(1,0){14}}
  \put(66,0){\line(1,0){14}}
  \put(43,3){\line(0,1){14}}
\put(103,0){\circle*{6}}
\put(86,0){\line(1,0){14}}
\put(123,0){\circle*{6}}
\put(106,0){\line(1,0){14}}
\put(143,0){\circle*{6}}
\put(126,0){\line(1,0){14}}
\end{picture} 
& 333 & 4 & 18 & 1\\

15 &
\begin{picture}(150,40)
  \put(3,0){\circle*{6}}
  \put(23,0){\circle{6}}
  \put(43,0){\circle*{6}}
  \put(63,0){\circle{6}}
  \put(83,0){\circle*{6}}
  \put(43,20){\circle{6}}
  \put(6,0){\line(1,0){14}}
  \put(26,0){\line(1,0){14}}
  \put(46,0){\line(1,0){14}}
  \put(66,0){\line(1,0){14}}
  \put(43,3){\line(0,1){14}}
\put(103,0){\circle{6}}
\put(86,0){\line(1,0){14}}
\put(123,0){\circle*{6}}
\put(106,0){\line(1,0){14}}
\put(143,0){\circle*{6}}
\put(126,0){\line(1,0){14}}
\end{picture} 
& 354 & 5 & 16 & 1\\

18 &
\begin{picture}(150,40)
  \put(3,0){\circle*{6}}
  \put(23,0){\circle*{6}}
  \put(43,0){\circle{6}}
  \put(63,0){\circle*{6}}
  \put(83,0){\circle{6}}
  \put(43,20){\circle*{6}}
  \put(6,0){\line(1,0){14}}
  \put(26,0){\line(1,0){14}}
  \put(46,0){\line(1,0){14}}
  \put(66,0){\line(1,0){14}}
  \put(43,3){\line(0,1){14}}
\put(103,0){\circle*{6}}
\put(86,0){\line(1,0){14}}
\put(123,0){\circle{6}}
\put(106,0){\line(1,0){14}}
\put(143,0){\circle*{6}}
\put(126,0){\line(1,0){14}}
\end{picture} 
& 397 & 5 & 14 & 1\\

20 &
\begin{picture}(150,40)
  \put(3,0){\circle*{6}}
  \put(23,0){\circle*{6}}
  \put(43,0){\circle{6}}
  \put(63,0){\circle*{6}}
  \put(83,0){\circle{6}}
  \put(43,20){\circle*{6}}
  \put(6,0){\line(1,0){14}}
  \put(26,0){\line(1,0){14}}
  \put(46,0){\line(1,0){14}}
  \put(66,0){\line(1,0){14}}
  \put(43,3){\line(0,1){14}}
\put(103,0){\circle*{6}}
\put(86,0){\line(1,0){14}}
\put(123,0){\circle*{6}}
\put(106,0){\line(1,0){14}}
\put(143,0){\circle*{6}}
\put(126,0){\line(1,0){14}}
\end{picture} 
& 438 & 7 & 12 & 1\\

24 &
\begin{picture}(150,40)
  \put(3,0){\circle*{6}}
  \put(23,0){\circle*{6}}
  \put(43,0){\circle{6}}
  \put(63,0){\circle*{6}}
  \put(83,0){\circle*{6}}
  \put(43,20){\circle*{6}}
  \put(6,0){\line(1,0){14}}
  \put(26,0){\line(1,0){14}}
  \put(46,0){\line(1,0){14}}
  \put(66,0){\line(1,0){14}}
  \put(43,3){\line(0,1){14}}
\put(103,0){\circle*{6}}
\put(86,0){\line(1,0){14}}
\put(123,0){\circle*{6}}
\put(106,0){\line(1,0){14}}
\put(143,0){\circle*{6}}
\put(126,0){\line(1,0){14}}
\end{picture} 
& 478 & 8 & 10 & 1\\

\hline
\end{longtable}
}
\end{proposition}
\section{Slice-induction and parametrization of orbits and classes}
\label{sec:slice-parametrization}
\subsection{Slice-induction}\label{sec:slice}

Theorem \ref{prop:local_holds} shows that the closure of a $G_0$-Jordan class in $V$ is a union of $G_0$-Jordan classes, generalising  results of \cite{BK,BH}. We aim at detecting which $G_0$-Jordan classes lie in the closure of a given one. In the classical case, this can be described in terms of Lusztig-Spaltenstein's parabolic induction of adjoint orbits, \cite{LS,Bo}. Slice induction is introduced in \cite{BH} to deal with the $m=2$ case, since orbit induction no longer works. We will briefly show how to extend to the case of general $m$ the construction in \cite{BH}, by combining some of its general arguments with our local approach. 

Let $\fm$ be now a $\theta$-stable reductive subalgebra of $\fg$ and $M$ the connected subgroup of $G$ with ${\rm Lie}(M)=\fm$. For a nilpotent element  $e\in \fm_1$ we consider a graded $\mathfrak{sl}(2)$-triple $\{e,h, f\}$ in $\fm$, so that $h\in\fm_0$ and $f\in\fm_{-1}$, and the corresponding Slodowy slice $S_{\fm,e}=e+\fm^f\subset\fm$. Since $\fm^f$ is homogeneous, we can consider its intersection with $V$, obtaining $S_{\fm_1,e}=e+\fm^f_1\subset\fm_1$. If $e=0$, we consider the trivial triple as an $\mathfrak{sl}(2)$-triple, so $S_{\fm_1,0}=\fm_1$. 
We start with two preliminary results in the case $\fm=\fg$.

\begin{lemma}\label{lem:transverse}
Let $\{e,h, f\}$ be a  graded  $\mathfrak{sl}(2)$-triple in $\fg$ and  let $X\subset V$ be an irreducible locally closed $G_0$-stable  subset such that $X\cap S_{\fg_1,e}\neq\emptyset$.
Then the action morphism $\psi\colon G_0\times S_{\fg_1,e}\to V$ is smooth, its restriction 
$\psi_X\colon G_0\times (S_{\fg_1,e}\cap X)\to X$ is smooth and dominant, more precisely $\psi_X(G_0\times C)$ is dense in $X$ for any irreducible component $C$ of  $S_{\fg_1,e}\cap X$. 
\end{lemma}
\begin{proof}
For $m=2$, this is part of \cite[Proposition 2.4 (i)]{BH}, we record the proof for completeness. The action morphism $\psi$ is $G_0$-equivariant with smooth domain and codomain, hence it suffices to verify that the differential is surjective at any point of the form $(1,y)\in G_0\times S_{\fg_1,e}$. 
We note that
\begin{equation*}
\begin{aligned}
d\psi|_{(1,y)}&:\fg_0\times \fg_1^f\to V\\
&(x,z)\to [x,y]+z
\end{aligned}
\end{equation*}
and by  $\mathfrak{sl}(2)$-representation theory  ${\fg}=[{\fg},e]\oplus\fg^f$, which in degree $1$ becomes
$V=[{\fg}_0,e]\oplus\fg_1^f$, so the differential at $(1,e)$ is surjective. The contracting ${\mathbb C}^*$-action argument in \cite[7.4, Corollary 1]{Slo}  carries over to the $\ZZ_m$-graded case because $\{e,h, f\}$ is a  graded  $\mathfrak{sl}(2)$-triple and $h\in\fm_0$, so $\psi$ is smooth at any point $(1,y)$, hence everywhere.

By \cite[III.5, Lemma 2]{Slo} applied to the $G_0$-equivariant morphism given by the inclusion of $X$ in $V$, the restriction $\psi_X$ is again smooth, whence open. Since $X$ is irreducible, the image of any irreducible component 
of  $G_0\times (S_{\fg_1,e}\cap X)$ is dense in $X$.
\end{proof}

\begin{lemma}\label{lem:nilpotent}
Let $J$ be a $G_0$-Jordan class in $V$ and $e\in \mathcal N_V$. Then $e\in \overline{J}$ if and only if $J\cap S_{\fg_1,e}\neq\emptyset$ if and only if $\overline{J}\cap S_{\fg_1,e}\neq\emptyset$.
\end{lemma} 
\begin{proof}
We note that $J$ is a locally closed $G_0$-stable cone by Proposition \ref{thm:fullcentralizer} and Corollary
\ref{cor:JClc}, so when $m=2$ these are the equivalences $(i)=(iv)=(v)$ in \cite[Theorem 2.6]{BH}. The proof of \cite[Lemma 2.3]{BH} shows the existence of a contracting ${\mathbb C}^*$-action on $S_{\fg_1,e}$ and it carries over to the $m>2$ case. If $J\cap S_{\fg_1,e}\neq\emptyset$ then each irreducible component of $\overline{J\cap S_{\fg_1,e}}$ is  stable under the $\mathbb C^*$-action, so $e$ lies in each of them.  As a consequence, $e\in\overline{J}$. 

Clearly $e\in\overline J$ gives $\overline{J}\cap S_{\fg_1,e}\neq\emptyset$, so it remains to show 
that  $\overline{J}\cap S_{\fg_1,e}\neq\emptyset$ implies  $J\cap S_{\fg_1,e}\neq\emptyset$. We follow the proof of \cite[Proposition 2.5]{BH}, establishing that 
$J\cap S_{\fg_1,e}$ is dense  in $\overline{J}\cap S_{\fg_1,e}$.  

Since $J$ is open in $\overline{J}$, the subset $J\cap S_{\fg_1,e}$ is open  in $\overline{J}\cap S_{\fg_1,e}$ and therefore
it is enough to prove that it meets every irreducible component $C$ of $\overline{J}\cap S_{\fg_1,e}$. The latter
follows then from the density of $G_0\cdot C$ in $\overline J$, guaranteed by  Lemma \ref{lem:transverse} applied to $X=\overline{J}$. 
\end{proof}

For the rest of the section we resume the convention from Subsection \ref{sec:local-study-closure} that whenever we have $y_s\in {\mathcal S}_V$, then $\fm:=\fg^{y_s}$,  and $M_0^\circ$ is the identity component of $M_0:=G_0^{y_s}=G_0\cap G^{y_s}$.

\begin{theorem}\label{thm:BH}
Let $J_1$, $J_2$ be $G_0$-Jordan classes in $V$. Then the following conditions are equivalent:
\begin{enumerate}
\item[$(i)$]\label{item:incluso} $J_2\subset\overline{J_1}$;
\item[$(ii)$]\label{item:inter} $J_2\cap\overline{J_1}\neq\emptyset$;
\item[$(iii)$]\label{item:slice} 
There exist $x\in J_1$, $y\in J_2$ such that  $\fg^{x_s}\subset \fm$ and
$J_{M_0^\circ}(x)\cap S_{\fm_1,y_n}\neq\emptyset$;
\item[$(iv)$]\label{item:closureM}  There exist $x\in J_1$, $y\in J_2$ such that  $\fg^{x_s}\subset \fm$ and
$y\in \overline{J_{M_0^\circ}(x)}.$
\end{enumerate}
\end{theorem} 
\begin{proof} This is the generalization of  \cite[Theorem 3.5]{BH} to the $m>2$ case, but our proof is slightly different and it combines 
Lemma \ref{lem:nilpotent} and Lemma \ref{lem:Um} with our local approach. 

The equivalence 
$(i)\Leftrightarrow (ii)$ is immediate from Theorem \ref{prop:local_holds}. We prove the other ones.
\vskip0.2cm\par\noindent
{\it Claim $(iii)\Leftrightarrow (iv)$}.
Lemma \ref{lem:nilpotent} applied to $\fm$, $y_n$ and $J_{M_0^\circ}(x)$ says that
$J_{M_0^\circ}(x)\cap S_{\fm_1,y_n}\neq\emptyset$ if and only if $y_n\in \overline{J_{M_0^\circ}(x)}$. Since
$y_s\in \fzeta(\fm)_1$, 
the latter condition is equivalent to $y\in \overline{J_{M_0^\circ}(x)}$
by Lemma \ref{lem:ss-closure} $(iv)$.
\vskip0.2cm\par\noindent
{\it Claim $(iv)\Rightarrow (ii)$}. Let $x,\,y$ be as in $(iv)$.

Since $\fg^x\subset\fg^{x_s}$,  we have $x\in U_{\fm}$ and hence $J_{G_0}(x)\cap U_{\fm}\cap J_{M_0^\circ}(x)\neq\emptyset$. 
Lemma \ref{lem:Um} $(iv)$ gives \begin{equation}
\label{eq:union-later-use}
{J_{G_0}(x)\cap U_{\fm}}= \bigcup_{i\in I_{J_1}}J_{M,i}\cap U_{\fm}
\end{equation}
and $J_{M_0^\circ}(x)$ is, by construction, one of the $M_0^\circ$-Jordan classes occurring in the R.H.S. Applying Proposition \ref{prop:inclusion-closure} to $J=J_{G_0}(x)$ and $y\in \overline{J_{M_0^\circ}(x)}$, we see that $y\in\overline{J_{G_0}(x)}$.
This shows $y\in J_2\cap\overline{J_1}\neq\emptyset$, proving $(ii)$.
\vskip0.2cm\par\noindent
{\it Claim $(ii)\Rightarrow (iv)$}. 
Assume now $y\in J_2\cap \overline{J_1}$. By Proposition \ref{prop:inclusion-closure}  the element $y$ lies in $\overline{J_{M,l}}$ where $J_{M,l}$ is an $M_0^\circ$-Jordan class such that 
$\emptyset\neq J_{M,l}\cap U_{\fm}\subset J_1\cap U_{\fm}$. Then, for any $x\in  J_{M,l}\cap U_{\fm}$ the pair $(x,y)$  satisfies $(iv)$. 
\end{proof}

Comparing dimensions of orbits in $J_1$ and $J_2$ we readily get:
\begin{corollary}\label{cor:reg-closure}Let $J_1$, $J_2$ be $G_0$-Jordan classes in $V$. Then $J_2\subset \overline{J_1}^\bullet$ if and only if there exist $x\in J_1$, $y\in J_2$ such that  $\fg^{x_s}\subset \fm$, $J_{M_0^\circ}(x)\cap S_{\fm_1,y_n}\neq\emptyset$ and $\dim\Oc_x^{M_0}=\dim \Oc_{y_n}^{M_0}$.
\end{corollary}

\begin{rem}
Condition $(iii)$ from Theorem \ref{thm:BH} is called weak slice-induction in \cite{BH}. If $J_2$ is weakly slice-induced from $J_1$ and satisfies the dimension condition in  Corollary \ref{cor:reg-closure}, then it is called slice-induced from $J_1$.  Slice-induction  is shown to coincide with parabolic induction in the ungraded case $m=1$ in \cite[Corollary 3.7]{BH}.
\end{rem} 

\begin{corollary}A $G_0$-Jordan class $J=J_{G_0}(y)$ contained in $V_{(d)}$ is dense in a sheet if and only if   $J_{M_0^\circ}(x)\cap S_{\fm_1,y_n}=\emptyset$ for any $x\in V_{(d)}\setminus J$ such that $\fg^{x_s}\subset\fm$.
\end{corollary}
\begin{proof}
First of all, the irreducible  subset $J$ is contained in some sheet $S$ in  $V_{(d)}$ and there is  a unique $G_0$-Jordan class $J'\subset V_{(d)}$ such that $S=\overline{J'}^\bullet$ by  Proposition \ref{prop:dense}.
The condition $J_{M_0^\circ}(x)\cap S_{\fm_1,y_n}=\emptyset$ for any $x\in V_{(d)}\setminus J$ such that $\fg^{x_s}\subset\fm$ is equivalent to say that there are no $G_0$-Jordan classes $\mathcal J\neq J$ such that $J\subset \overline{\mathcal J}^\bullet$, in other words, that $J=J'$.
\end{proof}
\subsection{Parametrization of orbits and classes}
\label{subsec:parametrization}
We aim at a parametrization of the $G_0$-orbits contained in a $G_0$-Jordan class
$J_{G_0}(x)=G_0\cdot(\fzeta(\fg^{x_s})^{reg}_1+x_n)$. 
By Theorem \ref{thm:472},  we may assume that  $x=x_s+x_n\in V$ with $x_s\in\fc$, so Corollary \ref{cor:serveanchedopo} ensures that $\fzeta(\fg^{x_s})_1\subset \fc$. Let 
\begin{equation*}
\Gamma:=N_{W_{Vin}}(\fzeta(\fg^{x_s})_1),
\end{equation*}
the stabilizer of  $\fzeta(\fg^{x_s})_1$ in $W_{Vin}$.
\begin{rem}\label{rem:Gamma}
\begin{enumerate}
\item Observe that $x_s\in \fc$ implies  $Z_{G_0}(\fc)\subset G_0^{x_s}\subset N_{G_0}(\fg^{x_s})$. Corollary \ref{cor:preliminaryfacts} gives also
$N_{G_0}(\fg^{x_s})=N_{G_0}(\fzeta(\fg^{x_s})_1)=N_{G_0}(\fzeta(\fg^{x_s}))$, so $\Gamma\cong \big(N_{G_0}(\fc)\cap N_{G_0}(\fg^{x_s})\big)/Z_{G_0}(\fc)$. In other words, if $w\in\Gamma$, then any of its representatives $\dot{w}\in N_{G_0}(\fc)$ lies in $N_{G_0}(\fg^{x_s})$.
\item  The group $N_{G_0}(\fc)\cap N_{G_0}(\fg^{x_s})$ normalizes $G^{x_s}_0$ and $\fg^{x_s}_1$ and thus
acts on the  set of $G^{x_s}_0$-orbits in $\fg^{x_s}_1$. Since $Z_{G_0}(\fc)\subset G_0^{x_s}$, this action factors through an action of $\Gamma$ on the set of 
$G^{x_s}_0$-orbits in $\fg^{x_s}_1$ which preserves the set of  nilpotent ones. Observe that $G^{x_s}_0$ and its orbits need not be connected. 
\end{enumerate}
\end{rem}

We shall need the stabilizer $\Gamma_n$  in $\Gamma$ of $\mathcal O_{x_n}^{G^{x_s}_0}$ with respect to the action defined above:
\begin{equation*}\Gamma_n={\rm Stab}_\Gamma(\mathcal O_{x_n}^{G^{x_s}_0}).\end{equation*}
\begin{proposition}
Let $x=x_s+x_n\in V$ with $x_s\in\fc$.
The assignment $\varphi$ from $\fzeta(\fg^{x_s})^{reg}_1$ to the orbit set $J_{G_0}(x)/G_0$ given by 
$y_s\mapsto \mathcal O^{G_0}_{(y_s+x_n)}$ induces
a homeomorphism  
$\overline{\varphi}\colon \fzeta(\fg^{x_s})^{reg}_1/ \Gamma_n\longrightarrow J_{G_0}(x)/G_0$, where the orbit set is endowed with the quotient topology. 
\end{proposition}
\begin{proof}
The map $\varphi$ is well-defined and surjective by Proposition \ref{eq:JordanIrr}. We prove injectivity. 
Let $y_s,z_s\in \fzeta(\fg^{x_s})^{reg}_1$ be such that $g\cdot(y_s+x_n)=z_s+x_n$ for some $g\in G_0$, i.e.,
\begin{align}
\label{eq:firststab}
g\cdot y_s&=z_s\;,\\
\label{eq:secondstab}
g\cdot x_n&=x_n\;,
\end{align}
and consider $w\in W_{Vin}$ such that $w\cdot y_s=z_s$, cf. Theorem \ref{theorem:weylvinberg}. Any representative $\dot{w}\in N_{G_0}(\fc)$ of $w$
satisfies $\dot{w}\cdot\fg^{x_s}=\dot{w}\cdot\fg^{y_s}=\fg^{z_s}=\fg^{x_s}$,
so $w\in \Gamma$ by Remark \ref{rem:Gamma}.
Moreover, $\dot{w}g^{-1}\in G^{z_s}\cap G_0=G^{x_s}_0$ 
by \eqref{eq:firststab}. It follows from \eqref{eq:secondstab} that
\begin{equation*}
\dot{w}\cdot x_n\in\mathcal O_{x_n}^{G^{x_s}_0}\qquad\text{so}\qquad \dot{w}\cdot \mathcal O_{x_n}^{G^{x_s}_0}=\mathcal O_{x_n}^{G^{x_s}_0}\;,
\end{equation*}
in other words $w\in \Gamma_n$ and $\overline{\varphi}$ is injective. 

Let $p\colon J_{G_0}(x)\to J_{G_0}(x)/G_0$ be the quotient map and $U$ an open subset in $J_{G_0}(x)/G_0$. Then $p^{-1}(U)$ is a $G_0$-stable open subset in $J_{G_0}(x)$ and its intersection 
$$
p^{-1}(U)\cap\big(\fzeta(\fg^{x_s})^{reg}_1\times\Oc_{x_n}^{G^{x_s}_0}\big)
$$
is an open $\Gamma_n$-stable subset of $\fzeta(\fg^{x_s})^{reg}_1\times\Oc_{x_n}^{G^{x_s}_0}$.
Its projection onto $\fzeta(\fg^{x_s})^{reg}_1$ is again an open $\Gamma_n$-stable subset, and so is its image through the quotient map by the finite group $\Gamma_n$.
We have therefore proved that $\overline{\varphi}$ is a continuous bijection, and it remains to show that is open.

By Corollary \ref{cor:JClc} and Proposition \ref{prop:Jordan-classes-are-smooth},
the action morphism $G_0\times (\fzeta(\fg^{x_s})^{reg}_1+x_n)\to J_{G_0}(x)$ is a morphism of smooth varieties whose
induced map on the tangent spaces is surjective.
Hence it is smooth, and an open morphism in the Zariski topology (see \cite[VII, Remark 1.2]{MR0274461} and \cite[V, Theorem 5.1 and VII, Theorem 1.8]{MR0274461}). From this, it is straightforward to see that $\overline{\varphi}$ is open. 
\end{proof}

We briefly turn to the parametrization of $G_0$-Jordan classes.
Thanks to Theorem \ref{theorem:weylvinberg} 
and Example \ref{rem:G0regular-in-V}
describing the centralizer of an element of $\fc$,
we easily establish the following.
\begin{lemma}
\label{lem:vinbergstrata}
Let $x_s$ and $y_s$ be two elements in $\fc$. Then the centralizers $\fg^{x_s}$ and $\fg^{y_s}$ are $G_0$-conjugate if and only if there exists $w\in W_{Vin}$ such that $w\cdot \Sigma(x_s)=\Sigma(y_s)$.
\end{lemma}
The hyperplane arrangement on $\fc$ determined by the restricted roots $\sigma\in\Sigma$ admits an action of $W_{Vin}$ and it induces a stratification on $\fc$, where two elements lie in the same stratum 
$
\fzeta(\fg^{x_s})^{reg}_1
=\left\{y_s\in\fc\mid \Sigma(y_s)=\Sigma(x_s) \right\}
$
if and only if their centralizers coincide. Equivalently, the stratum associated to a closed and symmetric subset $\widetilde\Sigma\subset \Sigma$ (as in of \cite[pag. 182]{MR1349140}) is
$$
S_{\widetilde\Sigma}=\left\{x\in\fc\mid\Sigma(x)=\widetilde\Sigma\right\}
$$
and the collection of $S_{\widetilde\Sigma}$'s is a finite partition of $\fc$. Already in the ungraded case, where the class of centralizers of semisimple elements coincides with the class of Levi subalgebras, not all closed and symmetric subsets $\widetilde\Sigma$ of $\Sigma$ give rise to a non-empty stratum. 
In the graded case, some information on stabilizers of generic elements in ${\mathcal S}_V$ is to be found in \cite{MR0430168} under the assumption that $V$ is a simple $G_0$-module.
We refer to \cite[Proposition 3.4]{MR3890218} for an alternative general description of centralizers of semisimple elements.
In view of Lemma \ref{lem:vinbergstrata}, two strata $S_{\widetilde\Sigma}$ and $S_{\widetilde\Sigma'}$ are said to be equivalent if $w\cdot\widetilde\Sigma=\widetilde\Sigma'$ for some $w\in W_{Vin}$. 
Given $\widetilde\Sigma\subset\Sigma$, we set $\fm(\widetilde\Sigma)$ to be the $\theta$-stable Levi subalgebra of $\fg$
constructed as in \eqref{eq:centralizerss} and $M(\widetilde\Sigma)\subset G$, $M(\widetilde\Sigma)_0=M(\widetilde\Sigma)\cap G_0\subset G_0$ as usual.

\begin{proposition}
\label{prop:bijective-Jordan-classes}
Jordan classes in $V$ are in one-to-one correspondence with $W_{Vin}$-classes of pairs $(\widetilde\Sigma,\mathcal O)$ where 
$\widetilde\Sigma\subset \Sigma$ satisfies $S_{{\widetilde\Sigma}}\neq\emptyset$ and $\mathcal O$ is a nilpotent orbit in $\fm(\widetilde\Sigma)_1$ for the action of $M(\widetilde\Sigma)_0$. 
\end{proposition}
\begin{proof}
Observe that   $N_{G_0}(\fc)$ acts on the set of pairs $(\widetilde\Sigma,\mathcal O)$ as above and that if  $\fm(\widetilde\Sigma)$ is the centralizer of some $x_s\in\fc$, then $Z_{G_0}(\fc)\subset M(\widetilde\Sigma)_0$, hence it acts trivially on $(\widetilde\Sigma,\Oc)$. Thus the action of $N_{G_0}(\fc)$ induces an action of $W_{Vin}$. 

Now recall that, for  $x\in V$ the assignment $J_{G_0}(x)\mapsto (\fg^{x_s},\Oc_{x_n}^{G_0^{x_s}})$ establishes a one-to one-correspondence between $G_0$-Jordan classes in $V$ and $G_0$-classes of pairs $(\mathfrak l,\Oc)$ where $\mathfrak l$ is the stabilizer of a semisimple element in $V$ and $\Oc$ a nilpotent orbit in $\mathfrak l_1$ for the action of $L_0$. Theorem \ref{thm:472} guarantees that we can always find a pair in the $G_0$-orbit where $\mathfrak l=\fm(\widetilde\Sigma)$ for some $\widetilde\Sigma\subset\Sigma$. 
Assume that for two pairs $(\fm(\widetilde\Sigma),\Oc)$ and $(\fm(\widetilde\Sigma'),\Oc')$ of this form  there is $g\in G_0$ such that
$(g\cdot\fm(\widetilde\Sigma),g\cdot\Oc)=(\fm(\widetilde\Sigma'),\Oc')$. By Lemma \ref{lem:vinbergstrata} we can decompose $g=g'\dot{w}$, where $g'\in N_{G_0}(\fm(\widetilde\Sigma'))$ and $\dot{w}\in N_{G_0}(\fc)$. In addition, $g'=l\dot{\sigma}$ with $l\in M(\widetilde\Sigma')_0$ and $\dot{\sigma}\in N_{G_0}(\fm(\widetilde\Sigma'))\cap N_{G_0}(\fc)$. In other words, we may replace $g$ by an element in $N_{G_0}(\fc)$, so 
$(\fm(\widetilde\Sigma),\Oc)$ and $(\fm(\widetilde\Sigma'),\Oc')$ lie in the same $W_{Vin}$-orbit. 
\end{proof}
The results of \cite{MR504529} encompass a parametrization of the $G_0$-Jordan classes, 
where $\theta$ is the automorphism of order $m=3$ of $\fg=E_8$ for which 
$\fg_{1}\cong\Lambda^3\CC^9$, $\fg_0\cong\fsl(9)$ and $\fg_{-1}\cong\Lambda^3(\CC^9)^*$ as 
in Example \ref{ex:tworegularityconditions}. This is shown in the following:
\begin{example}
\label{exampleE8-Jordan-classes}
By the discussion in \cite[\S 3.4]{MR504529}, the seven ``families'' described in \cite[\S 1]{MR504529}  parametrize the Levi subalgebras $\mathfrak l=\fg^{x_s}$ that arise from elements $x_s\in\mathcal S_V$ up to $G_0$-conjugation,
and the ``classes'' in Tables $1$-$6$ of \cite[\S 1]{MR504529} parametrize the nilpotent orbits in $\mathfrak l_1$ for the action of $G^{x_s}_0$. 
(If $x_s$ is in family I then $\fg^{x_s}=\fh$, there is no non-trivial nilpotent orbit and only one class.)
By Proposition \ref{prop:bijective-Jordan-classes}, our $G_0$-Jordan classes almost coincide with the classes of \cite{MR504529}: 
the finite group $N_{G_0}(\mathfrak l)/G^{x_s}_0$ acts on the set of nilpotent $G^{x_s}_0$-orbits in $\mathfrak l_1$, possibly glueing some of them.

Hence, some of the $164$ classes of \cite{MR504529} may correspond to the same $G_0$-Jordan class.
A look at Tables $1$-$6$ tells us that this may happen only in a few cases, since centralizers of elements of a $G_0$-Jordan class are $G_0$-conjugate by Proposition \ref{thm:fullcentralizer} and $N_{G_0}(\mathfrak l)/G^{x_s}_0=\1$ in the VII family:

\begin{itemize}
\item[] {\bf III family}: Classes $2$-$3$, $4$-$6$, and $7$-$8$;
\item[] {\bf V family}: Classes $7$-$8$, and $10$-$11$;
\item[] {\bf VI family}: Classes $5$-$6$, $8$-$9$, $11$-$12$, and $17$-$18$.
\end{itemize}

Recall that the support of a trivector $\varphi\in \Lambda^3\CC^9$
is the unique minimal subspace $E\subset\CC^9$ 
such that $\varphi\in \Lambda^3 E$. Its dimension is the {\it rank} of $\varphi$, one of the simplest discrete $G_0$-invariants of a trivector.
The nilpotent $G^{x_s}_0$-orbits associated to the classes $7$-$8$ in V family have different rank, so they are not $G_0$-related. Thus, they correspond to different $G_0$-Jordan classes. A similar observation works in all the remaining cases, except 
those of the III family and the classes $5$-$6$ of VI family, but it is not difficult to see that the nilpotent $G^{x_s}_0$-orbits of these last two classes are not $G_0$-related. 
It remains therefore to deal with the III family. 

First of all, the rank of the nilpotent orbit in class $4$ is strictly smaller than the rank of those in classes $5$ and $6$. However, the permutation matrix
\begin{equation}
\label{eq:matrix-appropriate}
g=-\left(\begin{array}{c|c|c} 
\Id_{3\times 3} & 0 & 0\\
\hline
0 & 0 & \Id_{3\times 3}\\
\hline
0 & \Id_{3\times 3} & 0
\end{array}
\right)
\end{equation}
is an element of $N_{G_0}(\mathfrak l)$ and it {\it does} relate the nilpotent $G^{x_s}_0$-orbits associated to classes $5$-$6$, which then correspond to a single $G_0$-Jordan class. The same is true for classes $2$-$3$ and $7$-$8$. 
In summary, the space 
$\Lambda^3\CC^9$ is partitioned into $161$ $G_0$-Jordan classes. 

The quotient $\Gamma/W_{x_s}$ of $\Gamma$ with the stabilizer $W_{x_s}$ of $x_s\in\fc$ in $W_{Vin}$ was found in  \cite[\S 3.4]{MR504529} for all families (see also the fourth and fifth columns of 
\cite[Table 7]{MR504529}). In the case of III family, it is a group of order $72$ generated by complex reflections. 
Consider, for example, the $G_0$-Jordan class III.5, represented by $x=x_s+x_n$. 
A simple check shows that $g$ as in \eqref{eq:matrix-appropriate} normalizes also $\fc$, so
$g\in N_{G_0}(\fzeta(\fg^{x_s}))\cap N_{G_0}(\fc)$ and, by our previous discussion, it is not in $\Gamma_n$. The $G_0$-orbits in the $G_0$-Jordan class III.5 are then parametrized by the quotient $\fzeta(\fg^{x_s})^{reg}_1/ \Gamma_n$ of $\fzeta(\fg^{x_s})^{reg}_1$ by a group $\Gamma_n$ of order $36$.

We conclude with an application of Theorem \ref{thm:BH}. Let $J_2=J_{G_0}(y)$ be the $G_0$-Jordan class numbered III.$7$, i.e., the one with the representative
$y=y_s+y_n$ given by
\begin{equation}
\label{eq:y_s-example}
y_s=\big(\be_{123}+\be_{456}+\be_{789}\big)+i\big(\be_{147}+\be_{258}+\be_{369}\big)\;,\quad
y_n=\be_{159}\;.
\end{equation}
The centralizer $\fm=\fg^{y_s}$ is a reductive Lie algebra with semisimple part $\fr$ of type $A_2\oplus A_2$. More precisely, 
the center of $\fm$ is $4$-dimensional and sits in degrees $\pm 1$: it consists of the two components in brackets that defines $y_s$ in \eqref{eq:y_s-example} and of their duals. The semisimple part $\fr=\fr_{-1}\oplus\fr_0\oplus\fr_1$ is graded as follows \cite[\S 2.4]{MR504529}:
\begin{equation}
\label{eq:centralizer-third-class}
\begin{aligned}
\fr_{1}&=\mathrm{span}\{\be_{159},\be_{267},\be_{348}\}\oplus \mathrm{span}\{\be_{168},\be_{249},\be_{357}\}\;,\\
\fr_0&=\mathrm{span}\{d_{159},d_{267},d_{348}\}\oplus\mathrm{span}\{d_{168},d_{249},d_{348}\}\;,\\
\fr_{-1}&=\mathrm{span}\{\be^{159},\be^{267},\be^{348}\}\oplus \mathrm{span}\{\be^{168},\be^{249},\be^{357}\}\;,
\end{aligned}
\end{equation}
where $\be^i$, for $1\leq i\leq 9$, is the dual basis of $(\mathbb C^9)^*$, $\be^{ijl}:=\be^i\wedge \be^j\wedge \be^l$ and the elements $d_{ijk}=[\be_{ijk},\be^{ijk}]$ satisfy $d_{159}+d_{267}+d_{348}=d_{168}+d_{249}+d_{348}=0$.
The direct sums of vector spaces in \eqref{eq:centralizer-third-class} correspond to the Lie algebra decomposition of $\fr$.

Let $J_1=J_{G_0}(x)$ be any of the $G_0$-Jordan classes in the II family, i.e., one of II.$1$, II.$2$ or II.$3$. The choice of representative $x=x_s+x_n$ given by
\begin{equation}
\label{eq:x_s-example}
x_s=y_s+\big(\be_{159}+\be_{267}+\be_{348}\big)\;,\quad
x_n=\begin{cases}\be_{168}+\be_{249}\;&\text{for II}.1\;,\\
\be_{168}\;&\text{for II}.2\;,\\
0&\text{for II}.3\;,
\end{cases}
\end{equation}
easily allows to check that $J_2\subset \overline J_1$. First of all $\fzeta(\fg^{x_s})_1$ is generated by the $3$ vectors in brackets in 
\eqref{eq:y_s-example} and \eqref{eq:x_s-example}, hence $y_s\in\fzeta(\fg^{x_s})_1$ and $\fg^{x_s}\subset \fm$.
A graded $\fsl(2)$-triple $\{e,h,f\}$ in $\fm$ with $e=y_n$ is provided by 
$f=\be^{159}$ and $h=d_{159}$, and the required Slodowy slice $S_{\fm_1,e}=e+\fm_1^f$ is the affine subspace in $\fm_1$ modeled on $\fm_1^f=\mathrm{span}\{\be_{267},\be_{348}\}\oplus \mathrm{span}\{\be_{168},\be_{249},\be_{357} \}\oplus\fzeta(\fm)_1$. It is evident that $x\in S_{\fm_1,e}$, so $J_2\subset \overline J_1$ thanks to Theorem \ref{thm:BH} (iii).
\end{example}

\appendix
\section{Cartan, Levi and parabolic subalgebras in $\mathbb Z_m$-graded Lie algebras}
\label{appendix:homogeneous-Levi}
\label{subsec:hcsa}
Let $\{\fg,\theta\}$ be a reductive $\mathbb Z_m$-graded Lie algebra and $\fc\subset V$ a fixed Cartan subspace. The existence of a homogeneous Cartan subalgebra $\fh$ of $\fg$ containing $\fc$ is a result
probably known to experts by a long time; the proof in \cite[\S 4.1]{reeder} is stated for $\fg$ simple, but its proof carries over for any reductive $\fg$.
\begin{proposition}
\label{thm:existencecartan}
There exists a homogeneous Cartan subalgebra $\fh=\bigoplus_{l\in\ZZ_m}\fh_{l}$ of $\fg$ that satisfies
$\fh\supset\fzeta(\fc_\fg(\fc))$ and $\fh_1=\fc$. 
\end{proposition} 
\begin{rem}
By \cite[\S 3.1]{MR0430168}, the Cartan subspace $\fc$ is {\it not} an algebraic subalgebra in general, unless $m\leq 2$. 
On the other hand $\fh$ 
and $\fzeta(\fc_\fg(\fc))$ are algebraic, hence $\fh\supset\fzeta(\fc_\fg(\fc))\supset \overline\fc$, where
$\overline\fc$ is the algebraic closure of $\fc$. It is clear that $\fh=\fzeta(\fc_\fg(\fc))$ if and only if 
$[\fc_\fg(\fc),\fc_\fg(\fc)]=0$ but we are not aware of any general condition under which
$\fzeta(\fc_\fg(\fc))=\overline\fc$.
\end{rem}
We will call {\it adapted} any Cartan subalgebra $\fh$ of $\fg$ as in Proposition \ref{thm:existencecartan}. For such an $\fh$, let
\begin{equation}
\label{eq:rootdecompositionfull}
\fg=\fh\oplus\bigoplus_{\alpha\in\Phi}\fg_{\alpha}
\end{equation}
be the root space decomposition of $\fg$ with respect to $\fh$, with associated set of roots $\Phi\subset\fh^*$. The automorphism $\theta:\fg\to\fg$ permutes the root spaces in \eqref{eq:rootdecompositionfull}:
\begin{lemma}
\label{lem:thetaonrootspaces}
For any $\alpha\in\Phi$, we have $\alpha\circ\theta\in\Phi$ and
$\theta^{-1}(\fg_\alpha)=\fg_{\alpha\circ\theta}$.
\end{lemma}
We note that any root $\alpha\in\Phi$ can be decomposed as $\alpha=\alpha_0+\alpha_1+\cdots+\alpha_{m-2}+\alpha_{m-1}$, where $\alpha_l|_{\fh_l}=\alpha|_{\fh_l}$ and $\operatorname{Ker}(\alpha_\ell)=\bigoplus_{k\neq l}\fh_{k}$ for any $l\in\mathbb Z_m$. Repeatedly applying Lemma \ref{lem:thetaonrootspaces}, we see that $$\alpha\circ\theta^l=\alpha_0+\omega^l\alpha_1+\cdots+(\omega^l)^{m-2}\alpha_{m-2}+(\omega^l)^{m-1}\alpha_{m-1}$$
is a root too, for any $l\in\mathbb Z_m$. In other words, we may consider the equivalence class of roots given by $[\alpha]=\left\{\alpha\circ\theta^l\mid l\in\mathbb Z_m\right\}$ for any $\alpha\in\Phi$.
We let $[\Phi]=\left\{[\alpha]\mid\alpha\in\Phi\right\}$ be the collection of such equivalence classes and note that the direct sum of root spaces
$$\fg_{[\alpha]}=\bigoplus_{l\in\mathbb Z_m}\fg_{\alpha\circ\theta^l}$$
is a homogeneous subspace of $\fg$, whence 
$\fg=\fh\oplus\bigoplus_{[\alpha]\in[\Phi]}\fg_{[\alpha]}$ is a decomposition of $\fg$ into homogeneous subspaces.

\vskip0.2cm\par
Now, the centralizer $\fg^x$ of any $x\in\fc$ is a homogeneous Levi subalgebra containing $\fc_\fg(\fc)$.  
A natural question is whether there exists a parabolic subalgebra of $\fg$ with Levi factor $\fg^x$ that is also homogeneous: we will now see that this is rarely the case.
For simplicity of exposition, we restrict to the case where $\fg$ is semisimple.

Let $\fp$ be a proper parabolic subalgebra of $\fg$ with a 
homogeneous Levi factor $\mathfrak l$ that contains $\fc_{\fg}(\fc)$. Then, there exists a $\mathbb Z$-grading 
\begin{equation}
\label{eq:Zgradingofg}
\fg=\bigoplus_{j\in\mathbb Z}\fg(j) 
\end{equation}
of $\fg$ such that $\fp=\fg(\geq 0)=\bigoplus_{j\geq 0}\fg(j)$ and $\mathfrak l=\fg(0)$. We let $Z\in\fg$ be the grading element of \eqref{eq:Zgradingofg}, the unique element in $\fg$ that satisfies $[Z,X]=jX$ for all $X\in\fg(j)$, $j\in\mathbb Z$, see, e.g., \cite{MR2532439}. 

Now $Z\in\mathfrak z(\fc_{\fg}(\fc))$, so it belongs to the adapted Cartan subalgebra
$\fh=\bigoplus_{l\in\ZZ_m}\fh_{l}$
of $\fg$ of Proposition \ref{thm:existencecartan}. We will write $Z=Z_0+\cdots+Z_{m-1}$, where $Z_l\in\fh_l$ for all $l\in\mathbb Z_m$.

\begin{definition}
Let $\alpha=\alpha_0+\cdots+\alpha_{m-1}\in\Phi$ be a root with respect to $\fh$ and $l\in\mathbb Z_m$. The $l^{\text{th}}$ {\it mode of $\alpha$} is the complex number $\lambda_l=\alpha_l(Z_l)$.
\end{definition}
We remark that $\alpha(Z)=\sum_{l\in\mathbb Z_m}\lambda_l$. Since the adjoint action of $Z$ has integer eigenvalues, we may apply Lemma \ref{lem:thetaonrootspaces} repeatedly to the roots $\alpha\circ\theta^l\in\Phi$ and get:
\begin{proposition}
The modes of $\alpha$ satisfy a system of linear equations of the form
\begin{equation}
\label{eq:cyclotomicmatrix}
\begin{pmatrix}
1 & 1 & 1 & \cdots & 1 \\
1& \omega & \omega^2 & \cdots & \omega^{m-1} \\
1 & \omega^2 & (\omega^2)^2 & \cdots & (\omega^{m-1})^{2}\\
\vdots & \vdots & \vdots &  & \vdots \\
1 & \omega^{m-1} & (\omega^2)^{m-1} & \cdots & (\omega^{m-1})^{m-1} 
\end{pmatrix}
\begin{pmatrix}
\lambda_0 \\
\lambda_1 \\
\lambda_2 \\
\vdots\\
\lambda_{m-1}
\end{pmatrix}
=
\begin{pmatrix}
n_0 \\
n_1 \\
n_2 \\
\vdots\\
n_{m-1}
\end{pmatrix}\;,
\end{equation}
where $n_l=\alpha(\theta^l(Z))\in{\mathbb Z}$ for any $l\in\{0,\,\ldots,\,m-1\}$.
\end{proposition}
The $m\times m$ matrix on the L.H.S. of \eqref{eq:cyclotomicmatrix} is a symmetric matrix of Vandermonde type with coefficients in the cyclotomic field $\mathbb Q(\omega)$. We denote it by $M(\omega)$ and compactly rewrite \eqref{eq:cyclotomicmatrix} as $M(\omega)\vec \lambda=\vec n$, where $\vec\lambda\in \mathbb C^m$ is the vector of modes and $\vec n\in\mathbb Z^m$. Clearly all modes are elements of $\mathbb Q(\omega)$, but we have the following stronger result for $\lambda_0$.
\begin{proposition}
\label{prop:non-existence}
The identity $m \lambda_0=\sum_{l\in\mathbb Z_m}n_l$ is satisfied for any $\alpha$, therefore $\lambda_0\in \tfrac{1}{m}{\mathbb Z}$. If  $\fh_0=0$, then $\fp$ is not $\theta$-stable.
\end{proposition}
\begin{proof} Let $W=\{\vec y\in\mathbb C^m|~\sum_{_{l\in\mathbb Z_m}} y_l=0\}$. All columns of $M(\omega)$ but the first one lie in $W$, so 
\begin{equation}\label{eq:vector}
\vec n-\lambda_0\begin{pmatrix}
1 \\
1 \\
1 \\
\vdots\\
1
\end{pmatrix}\in W.\end{equation} Adding all entries of the above vector 
 gives  $m\lambda_0=\sum_{l\in\mathbb Z_m}n_l\in {\mathbb Z}$. 
If $\fh_0=0$, then $Z_0=0$, so  $\lambda_0=0$ and
 $\vec n\in W\cap{\mathbb Z}^m$. 

Now, $\fh\subset \fg(0)$ and $\fp=\fh\oplus\bigoplus_{\alpha\in\Phi,
\alpha(Z)\geq0}\fg_\alpha$. If $\fg_\alpha\subset \bigoplus_{j>0}\fg(j)$, then $\alpha(Z)=n_0>0$ and, if $\fh_0=0$, there exists $l\in\mathbb Z_m$ such that $n_l<0$, i.e., $\theta^{-l}\fg_\alpha=\fg_{\alpha\circ\theta^l}\not\in \fp$. 
\end{proof}
\begin{example}
The $\mathbb Z_m$-graded Lie algebra $\{\fg,\theta,m\}=\{E_8,\theta,3\}$ as in Examples \ref{ex:tworegularityconditions} and 
\ref{exampleE8-Jordan-classes} satisfies $\fh_0=0$. By 
Theorem \ref{thm:472} and Proposition \ref{prop:non-existence}, {\it all} centralizers $\fg^x$ of non-zero $x\in\mathcal S_V$ do {\it not} extend to $\theta$-stable parabolic subalgebras.
\end{example}

\section*{Acknowledgments}
We thank Michael Bulois for useful email exchanges on sheets in symmetric Lie algebras and slice induction, Willem de Graaf for a helpful discussion on $\mathcal N$-regular gradings, Dmitri Panyushev for information on generic stabilizers of semisimple elements in $V$, and Vladimir Popov for pointing us reference \cite{MR3890218}.We thank the anonymous referees for pointing out an error in the earlier version of Corollary \ref{cor:mu} and for several comments, leading to an improved version of the manuscript. This research was partially supported by DOR1898721/18 and BIRD179758/17 funded by the University of Padova and FFABR2017 funded by MIUR. This project started when the third named author was holding a Type B Post doc Fellowship at the University of Padova. The third named author acknowledges that the research leading to these results has received funding from the Norwegian Financial Mechanism 2014-2021 (project registration number 2019/34/H/ST1/00636).
\bibliography{Biblio}
\bibliographystyle{plain}
%

\end{document}